\documentclass[11pt,reqno,a4paper]{amsart}
 \usepackage{amsfonts}
\usepackage{epsfig}
\usepackage{graphicx}
\usepackage{amsmath}
\usepackage{amssymb}
\usepackage{colordvi}
\usepackage{epsfig}
\usepackage{times,amsmath,epsfig,float,multicol,subfigure}

\usepackage{lipsum}

\usepackage{biblatex}

\usepackage{amsfonts, amssymb, amsmath, amscd, amsthm, txfonts, color, enumerate}

 \usepackage{graphicx}
\usepackage{hyperref}

\usepackage[T1]{fontenc}
\usepackage{amsmath}
\usepackage{amsfonts}
\usepackage{amssymb}
\usepackage{amsthm}
\usepackage{mathrsfs}
\usepackage{graphicx}
\usepackage{authblk}
\usepackage{mathtools}

\usepackage{tikz}
\usetikzlibrary{decorations.pathreplacing,calligraphy}

\usetikzlibrary{mindmap}
\tikzset{every node/.append style={scale=0.9}}
\usetikzlibrary{shapes.geometric, arrows}
\usetikzlibrary{positioning}

 \allowdisplaybreaks

\numberwithin{equation}{section}
\numberwithin{figure}{section}
\theoremstyle{plain}
\newtheorem*{thA}{Theorem A}
\newtheorem*{thB}{Theorem B}
\newtheorem{theorem}{Theorem}[section]
\newtheorem{proposition}[theorem]{Proposition}
\newtheorem{lemma}[theorem]{Lemma}
\newtheorem{corollary}[theorem]{Corollary}

\theoremstyle{definition}
\newtheorem{remark}[theorem]{Remark}

\newtheorem{definition}[theorem]{Definition}

\numberwithin{equation}{section}

\title[conditions on the continuous spectrum for ergodic Schrödinger Operators]{sufficient conditions on the continuous spectrum for ergodic Schrödinger Operators}

\author{Pablo Blas Tupac Silva Barbosa}
 
\address{Pablo Blas Tupac Silva Barbosa, Departamento de Matemáticas, UNAL, Universidad Nacional de Colombia - Sede Bogotá}
\email{psilvab@unal.edu.co}

\author{Rafael Alvarez Bilbao}

\address{Rafael Alvarez B., Escuela de Matem\'atica y Estat\'istica, UPTC, Sede Central del Norte Av. Central del Norte 39 - 115, cod. 150003 Tunja, Boyac\'a, Colombia, URL: https://orcid.org/0000-0001-8223-9434} 
	\email{rafael.alvarez@uptc.edu.co}

\begin{document}


\begin{abstract}
We study the spectral types of the families of discrete one-dimensional Schrödinger operators $\{H_\omega\}_{\omega\in\Omega}$, where the potential of each $H_\omega$ is given by $V_\omega(n)=f(T^n\omega)$ for $n\in\mathbb{Z}$, $T$ is an ergodic homeomorphism on a compact space $\Omega$ and $f:\Omega\rightarrow\mathbb{R}$ is a continuous function. We show that a generic operator $H_\omega\in \{H_\omega\}_{\omega\in\Omega}$ has purely continuous spectrum if $\{T^n\alpha\}_{n\geq0}$ is dense in $\Omega$ for a certain $\alpha\in\Omega$. We also show the former result assuming only that $\{\Omega, T\}$ satisfies topological repetition property (\textit{TRP}), a concept introduced by Boshernitzan and Damanik \cite{damanikBase}. Theorems presented in this paper weaken the hypotheses of the cited research and allow us to reach the same conclusion as those authors. We also provide 
a proof of Gordon's lemma, which is the main tool used in this work.
\end{abstract}

\keywords{} 
 
\subjclass[]{}

\date{\today}
\maketitle

\section{Introduction}
A Schrödinger operator is the Hamiltonian that describes the dynamics of a conservative system of particles at quantum scale in absence of relativistic forces. Despite its phenomenological origin, theory of Schrödinger operators stands as an autonomous branch of mathematics. Its study links notions of differential equations, geometric analysis, and measure theory, among other fields. In this paper we study the spectral types of discrete one-dimensional Schrödinger operators. This question is of interest because of its physical interpretation: in the system described by a Schrödinger operator, bound states are associated with presence of point spectrum of the operator, while scattering states are associated with continuous spectrum. Discrete one-dimensional Schrödinger operators are defined as:

\begin{equation}\label{Sec1Eq0}
	\begin{aligned}
		H: \mathcal{D}(H)\subseteq\ell^2(\mathbb{Z})&\rightarrow \ell^2(\mathbb{Z})\\
		\psi&\mapsto (\Delta_d+V)\psi
	\end{aligned}
\end{equation}

where $\Delta_d$ is the discrete Laplacian operator and $V$ is the potential function (formal definition of these operators is presented in section \ref{premilinaries}). The spectral theory of Schrödinger operators studies the spectrum of $H$ according to the properties of the function $V$.\\

Given a dynamical system $\{\Omega, T\}$ it is possible to define for each $\omega\in\Omega$ a potential function $V_\omega(n)=f(T^n\omega)$, where $f:\Omega \rightarrow \mathbb{R}$. This approach relates the theory of dynamical systems with the spectral theory of Schrödinger operators. In this context, the analysis focuses on studying the properties of the family of operators $\{H_\omega\}_{\omega\in\Omega}$, where each $H_\omega$ is given by:

\begin{equation}\label{Sec1Eq1}
	\begin{aligned}
		H_\omega: \mathcal{D}(H_\omega)\subseteq\ell^2(\mathbb{Z})&\rightarrow\ell^2(\mathbb{Z})\\
		\psi&\mapsto (\Delta_d+V_\omega)\psi.
	\end{aligned}
\end{equation}

If the transformation $T:\Omega\rightarrow \Omega$ is ergodic, then $\{H_\omega\}_{\omega\in\Omega}$ is a family of ergodic Schrödinger operators. The central question in this domain consists of two objectives: on the one hand, to determine the spectral types (point, absolutely continuous, and singular continuous) and the shape of the spectrum of the operator $H_\omega$ (equation \ref{Sec1Eq1}), and on the other hand, to describe the dynamics of physical systems associated with $H_\omega$. A particular interest in the theory of ergodic Schrödinger operators lies in the study of properties that are satisfied for a generic element of $\{H_\omega\}_{\omega\in\Omega}$. The latter implies, from the topological point of view, studying the properties of an element in a residual subset of $\{H_\omega\}_{\omega\in\Omega}$, and from the measure-theoretical approach, studying the properties of an element in a subset of full measure of $\{H_\omega\}_{\omega\in\Omega}$. In this paper we study the spectral properties of a generic $H_\omega\in\{H_\omega\}_{\omega\in\Omega}$ from the topological point of view. \\

The main background of this work is the research of Boshernitzan and Damanik \cite{damanikBase}. They introduce the definition of topological and metric repetition property (\textit{TRP} and \textit{MRP}, respectively) on the system $\{\Omega, T\}$, and study its implications on the absence of point spectrum of $\{H_\omega\}_{\omega\in\Omega}$, where the potential $V_\omega(n)=f(T^n\omega)$ is given by an ergodic homeomorphism $T$ on a compact space $\Omega$. Another related research is the work of Avila and Damanik \cite{avilaDamanik}. These authors show, using tools from harmonic analysis and Kotani Theory, that the absense of absolute continuous spectrum is a generic property of $\{H_\omega\}_{\omega\in\Omega}$, where $T:\Omega \rightarrow \Omega$ is a nonperiodic homeomorphism.\\

The aim of this paper is to determine sufficient conditions for the purely continuous spectrum of a family of ergodic Schrödinger operators $\{H_\omega\}_{\omega\in\Omega}$ to be a generic property. Our strategy is to show that a generic operator $H_\omega\in\{H_\omega\}_{\omega\in\Omega}$ has no eigenvalues (and therefore its point spectrum is empty). This objective is similar to that of \cite{damanikBase} and \cite{avilaDamanik}, since we seek to rule out the presence of a certain spectral type in a generic Schrödinger operator. 
Our work is based on two theoretical tools:

\begin{itemize}
	\item[\textit{(i)}] \textit{Gordon's Lemma}: this is a classical postulate of spectral theory that gives conditions for the absence of point spectrum of a Schrödinger operator.
 
\item[\textit{(ii)}] \textit{TRP}: this property provides the system $\{\Omega, T\}$ with a structure that allows, using Gordon's lemma, to show that a generic operator $H_\omega$ of the family $\{H_\omega\}_{\omega\in\Omega}$ has purely continuous spectrum (see definition \ref{repeticionTopyMet}).
\end{itemize}

Our results extend the main theorem of Boshernitzan and Damanik (\cite{damanikBase}, p.650). These authors demonstrate that purely continuous spectrum is a generic property of $\{H_\omega\}_{\omega\in\Omega}$ using two hypotheses: that the system $\{\Omega, T\} $ is minimal and satisfies \textit{TRP}. We weaken the hypotheses of this result from two perspectives: first, we prove the previous result using only the hypothesis that $\{T^n\alpha\}_{n\geq0}$ is dense in $\Omega$ for certain element $\alpha\in \Omega$ (specifically, for $\alpha\in PRP(T)$. \textit{See definition \ref{PRP}}). Then we demonstrate that \textit{TRP} is sufficient condition for the purely continuous spectrum to be a generic property of the family of ergodic Schrödinger operators $\{H_\omega\}_{\omega\in\Omega}$. The following are the statements of the two main results of this paper.
\begin{thA}
\label{thA}
Suppose that $\alpha\in PRP(T)$ and $\{T^n\alpha\}_{n\geq0}$ is dense in $\Omega$. Then there exists a residual subset $\mathcal{F}$ of $C(\Omega)$ such that if $f\in\mathcal{F}$ then there exists a residual subset $\Omega_f$ of $\Omega$ with the property that $f(T^n\omega)$ is a Gordon potential, for every $\omega\in \Omega_f$. 
\end{thA}

\begin{thB}
\label{thB}
Suppose that the dynamical system $\{\Omega, T\}$ satisfies \textit{TRP}. Then there exists a residual subset $\mathcal{F}$ of $C(\Omega)$ such that if $f\in\mathcal{F}$ then there exists a residual subset $\Omega_f$ of $\Omega$ with the property that $f(T^n\omega)$ is a Gordon potential, for every $\omega\in \Omega_f$.
\end{thB}

In addition to this introduction, the paper is divided into five sections. In the second section we introduce the main definitions in which the work is framed. In the third, we state and demonstrate Gordon's lemma. In the fourth and fifth sections, we prove theorems A and B, respectively. Finally, in the sixth section, we discuss some applications.

\section{Setting and main results}
\label{premilinaries}

In this section we introduce the formal definition of discrete Schrödinger operators. We also present the definition of \textit{TRP} property and discuss some of its implications on the structure of a dynamical system $\{\Omega, T\}$. 

\subsection{Schrödinger operators.} Let $ \mathscr{L}^{2}(\mathbb{R}^n)$ be the space of square integrable functions, which is defined

\begin{equation*}
		 \mathscr{L}^2(\mathbb{R}^n)=\{\Psi:\mathbb{R}^n\rightarrow \mathbb{C} \mid \int_{\mathcal{D}(\Psi)}\lVert \Psi(q)\rVert^2 dq<\infty\}
\end{equation*}

where the inner product is given by:

\begin{equation*}
		\langle \Psi,\Phi\rangle=\int_{\mathcal{D}}\overline{\Psi(q)}\Phi(q)dq.
	\end{equation*}
Discrete version of $ \mathscr{L}^{2}(\mathbb{R}^n)$ is defined as:

\begin{equation*}
		\ell^2(\mathbb{Z}^n)=\{\psi:\mathbb{Z}^n\rightarrow\mathbb{C}\mid \sum_{n\in\mathbb{Z}^n}|\psi(n)|^2<\infty\}
\end{equation*}
 
endowed with inner product:
\begin{equation*}
		\langle \psi,\varphi\rangle=\sum_{n\in\mathbb{Z}}\overline{\psi(n)}\varphi(n).
\end{equation*}
 
\begin{definition} 
\label{operdor de schrodinger}
A Schrödinger operator is defined as a linear operator

\begin{equation}\label{Kap2Eq5}
		\begin{aligned}
			\mathcal{H}: \mathcal{D}(\mathcal{H})\subseteq \mathscr{L}^2(\mathbb{R}^n)&\rightarrow  \mathscr{L}^2(\mathbb{R}^n)\\			
			\Psi&\mapsto (\Delta+V)\Psi
		\end{aligned}
	\end{equation}
\end{definition}
 where $\Delta$ is the Laplacian in $ \mathscr{L}^2(\mathbb{R}^n)$ and $V:\mathbb{R}^n\rightarrow\mathbb{R}$ the potential function. Similary, discrete Schrödinger operators are defined as:

\begin{equation}\label{Kap2Eq8}
		\begin{aligned}
			H: \mathcal{D}(H)\subseteq\ell^2(\mathbb{Z}^n)&\rightarrow \ell^2(\mathbb{Z}^n)\\			
			\psi&\mapsto (\Delta_d+V)\psi
		\end{aligned}
	\end{equation}
	with $\Delta_d$ the discrete Laplacian in $\ell^2(\mathbb{Z}^n)$ and $V:\mathbb{Z}^n\rightarrow\mathbb{R}$ the potential function.
	
	\quad
	
	In this paper we focus on the one-dimensional case, which leads to the following definition, considering $\Delta_d = \psi(n+1)+\psi(n-1)$ as the discrete Laplacian in $\ell^2(\mathbb{Z})$.
	
	\begin{definition}
	 \label{Operador de Schrödinger unidimensional discreto}
	 The discrete one-dimensional Schrödinger operator:
	 
	 \begin{equation}\label{Kap2Eq9}
		\begin{aligned}
			H: \mathcal{D}(H)\subseteq\ell^2(\mathbb{Z})&\rightarrow \ell^2(\mathbb{Z})\\			
			\psi&\mapsto (\Delta_d+V)\psi\\&=\psi(n+1)+\psi(n-1)+V(n)\psi(n)
		\end{aligned}
	\end{equation}
	\end{definition}

\begin{proposition}
\label{proposition1}
The operator $H$ (equation \ref{Kap2Eq9}) is self-adjoint.
\end{proposition}

\begin{proof}

	Let $\psi(n)$ and $\varphi(n)\in\ell^2(\mathbb{Z})$. Then, on the one hand:
	\begin{equation}\label{Kap2Eq91}
		\begin{aligned}
			\langle H\psi(n),\varphi(n)\rangle&=\langle (\Delta_d+V)\psi(n),\varphi(n)\rangle &&\text{by definition of $H$}\\
			&=\langle \Delta_d\psi(n),\varphi(n)\rangle+\langle V(n)\cdot\psi(n),\varphi(n)\rangle&&\text{by definition of $\langle,\rangle$}\\
			&=\langle \psi(n),\Delta_d\varphi(n)\rangle+\langle V(n)\cdot\psi(n),\varphi(n)\rangle&&\text{as $\Delta_d$ is self-adjoint}
		\end{aligned}
	\end{equation}
	On the other hand, note that:
	\begin{equation}\label{Kap2Eq10}
		\begin{aligned}
			\langle V(n)\cdot\psi(n),\varphi(n)\rangle&=\sum_{n\in\mathbb{Z}}\overline{V(n)\psi(n)}\varphi(n)&&\text{by definition of $\langle,\rangle$}\\
			&=\sum_{n\in\mathbb{Z}}\overline{\psi(n)V(n)}\varphi(n)\\
			&=\sum_{n\in\mathbb{Z}}\overline{\psi(n)}V(n)\varphi(n)&&\text{since $V(n)\in\mathbb{R}$}\\
			&=\langle \psi(n),V(n)\cdot\varphi(n)\rangle&&\text{by definition of $\langle,\rangle$}
		\end{aligned}
	\end{equation}
	Replacing \eqref{Kap2Eq10} in \eqref{Kap2Eq9} we obtain:
	\begin{equation*}
		\begin{aligned}
			\langle H\psi(n),\varphi(n)\rangle&=\langle \psi(n),\Delta_d\varphi(n)\rangle+\langle \psi(n),V(n)\varphi(n)\rangle\\&=\langle \psi(n),(\Delta_d+V)\varphi(n)&&\text{by definition of $\langle,\rangle$}\\
			&=\langle\psi(n),H\varphi(n)\rangle&&\text{for definition $H$}
		\end{aligned}
	\end{equation*}
	we conclude that $H$ is self-adjoint.
 \end{proof}
As a consequence of the previous proposition we have the next corollary.

\begin{corollary}
\label{corollary1}
 Let $H=\Delta_d + V$ be a discrete one-dimensional Schrödinger operator. If $V$ is bounded, then $H$ is a bounded self-adjoint operator. 
\end{corollary}



\subsection{Discrete one-dimensional ergodic Schrödinger operators.}

Let $\{\Omega, T\}$ be a dynamical system and consider $f:\Omega\rightarrow\mathbb{R}$. Each $\omega\in \Omega$ and its orbit over $T$ (\textit{i.e.} the set $\{T^n\omega\}_{n\in\mathbb{Z}}$) induces the definition of a potential function:
\begin{equation}\label{Kap2Eq11}
	\begin{aligned}
		V_\omega:\mathbb{Z}&\rightarrow\mathbb{R}\\
		n&\mapsto f(T^n\omega).
	\end{aligned}
\end{equation}

Consequently, each $\omega\in\Omega$ is associated with a Schrödinger operator
\begin{equation}\label{Kap2Eq12}
	\begin{aligned}
		H_\omega: \mathcal{D}(H_\omega)\subseteq\ell^2(\mathbb{Z})&\rightarrow\ell^2(\mathbb{Z})\\
		\psi&\mapsto (\Delta_d+V_\omega)\psi.
	\end{aligned}
\end{equation}
and $\{H_\omega\}_{\omega\in\Omega}$ is called \textit{family of Schrödinger operators with dynamically defined potential}.\\

We denote by $\mathcal{S}(\mathcal{H})$ 
the space of self-adjoint operators in a Hilbert space, that is:
\begin{equation*}
	\mathcal{S}(\mathcal{H})=\{H:\mathcal{D}(H)\subseteq\mathcal{H}\rightarrow\mathcal{H}\mid \text{$H$ is self-adjoint}\}
\end{equation*}
The following definition relates the ergodic theory to the Schrödinger operators. 

\begin{definition}\label{ergodic schrodiger operator}
    Let $(\Omega,\mathcal{B},\mu)$ be a probability space and consider:
	\begin{equation*}
		\begin{aligned}
			A:\Omega&\rightarrow\mathcal{S}(\mathcal{H})\\
			\omega&\mapsto H_\omega
		\end{aligned}
	\end{equation*} 
	The operator $H_\omega$ is an \textit{ergodic operator}, if there exists a family of ergodic transformations $\{T_i\}_{i\in\mathbb{Z}}$ in $\Omega$ and a family of unitary operators $\{U_i\}_{i\in\mathbb{Z}}$ en $\mathcal{H}$
such that:
	\begin{equation*}
		H_{T_i(\omega)}=U_i^* H_{\omega}U_i
	\end{equation*}
	where $U_i^*$ is the adjoint operator of $U_i$.
\end{definition}

\begin{lemma}
\label{lema22-1}
Let $(\Omega, \mathcal{B},\mu)$ be a probability space, $\{\Omega, T\}$ a dynamical system with $T$ a $\mu-$ergodic transformation and $f:\Omega\rightarrow\mathbb{R}$ a given function. Then, the operator $H_\omega$ (equation \ref{Kap2Eq12}) is ergodic.
\end{lemma}

\begin{proof}
Consider the transformation:
	\begin{equation*}
		\begin{aligned}
			A: \Omega&\rightarrow\mathcal{S}(\mathcal{H})\\
			\omega&\mapsto H_\omega=\Delta_d+V_\omega
		\end{aligned}
	\end{equation*}
	By hypothesis $T$ is $\mu-$ergodic, so $\{T^t\}_{t\in\mathbb{Z}}$ is a family of ergodic transformations on $\Omega$. To prove that $H_\omega$ is an ergodic operator, it is necessary to find a family of unitary operators $\{U_t\}_{t\in\mathbb{Z}}$ such that:
	\begin{equation}\label{Kap2Eq13}
		H_{T^t(\omega)}=U_t^* H_{\omega}U_t
	\end{equation}
	 For each $t\in\mathbb{Z}$, consider the translation operator on $\ell^2(\mathbb{Z})$:
	\begin{equation*}
		\begin{aligned}
			U_t: \ell^2(\mathbb{Z})&\rightarrow \ell^2(\mathbb{Z})\\
			\psi(n)&\mapsto\psi(n+t)
		\end{aligned}	
	\end{equation*}
	Notice that the operator $U_t$ is unitary:
	\begin{equation*}
		\begin{aligned}
			\lVert U_t\rVert=\sup_{\psi\neq0}\frac{\lVert U_t \psi(n)\rVert}{\lVert\psi(n)\rVert}=\sup_{\psi\neq0}\frac{\sum_{n\in\mathbb{Z}}|\psi(n+t)|^2}{\sum_{n\in\mathbb{Z}}|\psi(n)|^2}=\sup_{\psi\neq0}\frac{\sum_{n\in\mathbb{Z}}|\psi(n)|^2}{\sum_{n\in\mathbb{Z}}|\psi(n)|^2}=1
		\end{aligned}
	\end{equation*}
	therefore, $\{U_t\}_{t\in\mathbb{Z}}$ is a family of unitary operators on $\ell^2(\mathbb{Z})$. Additionally, we have that:
	\begin{equation*}
		U_t^*=U_{-t},\quad\forall t\in\mathbb{Z}
	\end{equation*}
	Let $\psi(n)\in\ell^2(\mathbb{Z})$. then:
	\begin{equation}\label{Kap2Eq14}
		\begin{aligned}
			H_{T^t(\omega)}\psi(n)=(\Delta_d+V_{T^t\omega})\psi(n)&=\Delta_d \psi(n)+V_{T^t\omega}\psi(n)\\
			&=\psi(n+1)+\psi(n-1)+f(T^{n}[T^t\omega])\psi(n)\\
			&=\psi(n+1)+\psi(n-1)+f(T^{n+t}\omega)\psi(n)
		\end{aligned}
	\end{equation}
	furthermore:
	\begin{equation}\label{Kap2Eq15}
		\begin{aligned}
			U_{-t}H_\omega U_t \psi(n)=U_{-t}H_\omega \psi(n+t)&=U_{-t}[(\Delta_d+V_\omega)\psi(n+t)]\\
			&=U_{-t}[\Delta_d \psi(n+t)+V_\omega\psi(n+t)]\\
			&=U_{-t}(\psi(n+t+1)+\psi(n+t-1)+f(T^{n+t}\omega)\psi(n+t))\\
			&=\psi(n+1)+\psi(n-1)+f(T^{n+t}\omega)\psi(n)
		\end{aligned}
	\end{equation}	
	
	from equations \eqref{Kap2Eq14} and \eqref{Kap2Eq15}:
	\begin{equation*}
		H_{T^t(\omega)}\psi(n)=U_t^* H_{\omega}U_t\psi(n),\quad\forall\psi(n)\in\ell^2(\mathbb{Z})
	\end{equation*}
\end{proof}

Lemma \ref{lema22-1} allows us to conclude that $H_\omega$ (equation \ref{Kap2Eq12}) is an ergodic operator. Therefore, if $T$ is an ergodic homeomorphism on $\Omega$, then $\{H_\omega\}_{\omega\in\Omega}$ is a family of discrete one-dimensional Schrödinger ergodic operators.

\subsection{Repetition property and the dynamical system \texorpdfstring{$\{\Omega, T\}$}{(O,T)} }

\label{sec: espectro continuo topología}

\begin{definition}
 \label{repetición}
Let $(\Omega, d)$ be a compact metric space. The sequence $\{\omega_n\}_{n\geq0}\subseteq\Omega$ satisfies \textit{the repetition property} (\textit{RP}) if for all $\varepsilon>0$ and $r\in\mathbb{Z}_+$ there exist $q\in\mathbb{Z}_+$ such that $d(\omega_n,\omega_{n+q})<\varepsilon$ for all $n\in\{0,1,\ldots,rq\}$.
\end{definition}
Figure \ref{figure1:PR} illustrates \textit{RP} for $\varepsilon>0$ with parameter $r=4$.
\begin{figure}[!ht]%
\label{figure1:PR}
\raggedright
\begin{tikzpicture}
	\draw [dashed](0,0) node[circle,draw, inner sep=1pt, label=below:$\omega_1$](z0){}
	circle (2);
	\draw[-stealth] (z0) -- (150:2) node[midway,above]{$\varepsilon$};
	\node [black] at (1.2,-1.2) {\textbullet};
	\node [black] at (0.9,-1.4) {$\omega_{1+q}$};
	\node [darkgray] at (-1.2,-0.5) {\textbullet};
	\node [black] at (-1.1,-0.8) {$\omega_{1+2q}$};
	\node [gray] at (0.3,0.4) {\textbullet};
	\node [black] at (0.9,0.4) {$\omega_{1+3q}$};
	\node [black] at (0,-2.5) {$n=1$};
	\hspace{4.5cm}
	\draw [dashed](0,0) node[circle,draw, inner sep=1pt, label=below:$\omega_2$](z0){}
	circle (2);
	\draw[-stealth] (z0) -- (150:2) node[midway,above]{$\varepsilon$};
	\node [black] at (1.2,-1.2) {\textbullet};
	\node [black] at (0.9,-1.4) {$\omega_{2+q}$};
	\node [darkgray] at (-1.2,-0.5) {\textbullet};
	\node [black] at (-1.1,-0.8) {$\omega_{2+2q}$};
	\node [gray] at (0.3,0.4) {\textbullet};
	\node [black] at (0.9,0.4) {$\omega_{2+3q}$};
	\node [black] at (0,-2.5) {$n=2$};
	\hspace{2.5cm}
	$\cdots$
	\hspace{2.5cm}
	\draw [dashed](0,0) node[circle,draw, inner sep=1pt, label=below:$\omega_{q-1}$](z0){}
	circle (2);
	\draw[-stealth] (z0) -- (150:2) node[midway,above]{$\varepsilon$};
	\node [black] at (1.2,-1.2) {\textbullet};
	\node [black] at (0.9,-1.4) {$\omega_{2q-1}$};
	\node [darkgray] at (-1.2,-0.5) {\textbullet};
	\node [black] at (-1.1,-0.8) {$\omega_{3q-1}$};
	\node [gray] at (0.3,0.4) {\textbullet};
	\node [black] at (0.9,0.4) {$\omega_{4q-1}$};
	\node [black] at (0,-2.5) {$n=q-1$};
\end{tikzpicture}
\caption{{\small Graphic representation of property \textit{RP}.}} \label{fig:PR}
\end{figure}
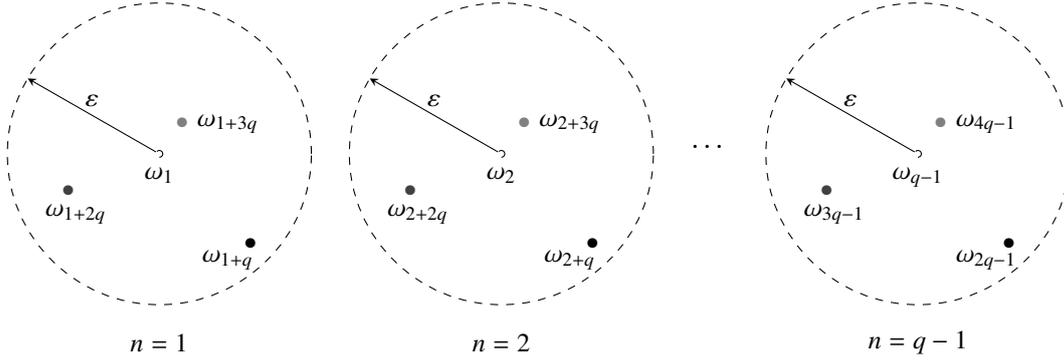

\hspace{2cm}

Property \textit{RP} is independent of the metric defined on $\Omega$ if the latter is compact. This is demonstrated in the following lemma.
\begin{lemma}
\label{métricas equivalentes}
	Let $(\Omega, d_1)$ and $(\Omega, d_2)$ be compact metric spaces and suppose that $\{\omega_n\}_{n\geq0}$ satisfies \textit{RP} on $(\Omega, d_1)$. Then, $\{\omega_n\}_{n\geq0}$ satisfies \textit{RP} on $(\Omega, d_2)$.
\end{lemma}
\begin{proof}

	Suppose that $\{\omega_n\}_{n\geq0}$ satisfies \textit{PR} on $(\Omega, d_1)$. Let $\varepsilon=\frac{1}{k}$ and $r=k$ for $k\in\mathbb{Z}_+$ in definition \ref{repetición}. Consider $B_{d_1}(\omega_i, \frac{1}{k})$ the open ball with center in $\omega_i$ and radius $\frac{1}{k}$ on the space $(\Omega, d_1)$. Observe that:
	\begin{equation*}
		\bigcup_{i=0}^{kq}B_{d_1}(\omega_i, \frac{1}{k}) \cup A
	\end{equation*}
	is a finite covering of $\Omega$ on the space $(\Omega, d_1)$, where $A$ is a open set on $(\Omega, d_1)$ that cover to $\Omega\setminus\{\omega_0,\ldots, \omega_{kq}\}$. In the same way:
 
	\begin{equation*}	\bigcup_{i=0}^{\overline{k}\overline{q}}B_{d_2}(\omega_i, \frac{1}{\overline{k}}) \cup B
	\end{equation*}
	is a finite covering of $\Omega$ on the space $(\Omega, d_2)$, where $B_{d_2}(\omega_i, \frac{1}{\overline{k}})$ is the open ball with center at $\omega_i$ and radius $\frac{1}{\overline{k}}$, and $B$ is a open set that covers $\Omega\setminus\{\omega_0,\ldots, w_{\overline{k}\overline{q}}\}$ on the space $(\Omega, d_2)$. \\
 
Given $\overline{k}\in\mathbb{Z}_+$, there exist $k\in\mathbb{Z}_+$ such that $\frac{1}{k}<\frac{1}{\overline{k}}$. Since $B_{d_1}$ and $B_{d_2}$ are open balls on $(\Omega, d_1)$ and $(\Omega, d_2)$ respectively, for every $x\in\Omega$ we have that $B_{d_1}(x,\frac{1}{k})\subset B_{d_2}(x,\frac{1}{\overline{k}})$ or $B_{d_2}(x,\frac{1}{\overline{k}})\subset B_{d_1}(x,\frac{1}{k})$. We will examine each of these two  cases.\\
	
	First, we assume that $B_{d_1}(x,\frac{1}{k})\subset B_{d_2}(x,\frac{1}{\overline{k}})$ and consider $x=\omega_n$, for $n\in\{0,\ldots, kq\}$. By hypotheses $\{\omega_n\}_{n\geq0}$ satisfies \textit{RP} on $(\Omega, d_1)$, which mean that $\omega_{n+q}\in B_{d_1}(\omega_n,\frac{1}{k})$ for all $n\in\{0,\ldots, kq\}$. As a result: 
	\begin{equation*}
		\omega_{n+q}\in B_{d_1}(\omega_n,\frac{1}{k})\subset B_{d_2}(\omega_n,\frac{1}{\overline{k}}) \quad\Rightarrow\quad\omega_{n+q}\in B_{d_2}(\omega_n,\frac{1}{\overline{k}}), \quad \forall n\in\{0,\ldots, kq\},
	\end{equation*}
	taking, $\overline{q}=q$ we conclude that $\{\omega_n\}_{n\geq0}$ satisfies \textit{RP} on $(\Omega, d_2)$.\\
	
	Now we assume that $B_{d_2}(x,\frac{1}{\overline{k}})\subset B_{d_1}(x,\frac{1}{k})$. This implies, for $x=\omega_n$ with $n\in\{0,\ldots, kq\}$, that if $d_1(\omega_n,y)<\frac{1}{k}$, then $d_2(\omega_n,y)<\frac{1}{k}$. By hypotheses $\{\omega_n\}_{n\geq0}$ satisfies \textit{RP} on $(\Omega, d_1)$, therefore $d_1(\omega_n,\omega_{n+q})<\frac{1}{k}$ for $n\in\{0,\ldots,kq\}$. Then:
	\begin{equation*}
		\begin{aligned}
			B_{d_2}(\omega_n,\frac{1}{\overline{k}})\subset B_{d_1}(\omega_n,\frac{1}{k})\quad&\Rightarrow\quad d_2(\omega_n,\omega_{n+q})<d_1(\omega_n,\omega_{n+q})<\frac{1}{k}\\
			&\Rightarrow \quad d_2(\omega_n, \omega_{n+q})<\frac{1}{k}, \quad \forall n\in\{0,\ldots, kq\}
		\end{aligned}
	\end{equation*}
	then $\{\omega_n\}_{n\geq0}$ satisfies \textit{RP} on $(\Omega, d_2)$.
\end{proof}
	\quad

Now we define the set \textit{PRP(T)}, which relates property \textit{RP} with dynamical systems.

\begin{definition}
\label{PRP}
	Let $\{\Omega,T\}$ be a dynamical system. $PRP(T)$ is defined as the set of points in $\Omega$ such that $\{T^n\omega\}_{n\geq0}$ satisfies \textit{RP}.
	\begin{equation}\label{Kap4Eq2}
		PRP(T)=\{\omega\in\Omega \mid \forall \varepsilon>0, r\in\mathbb{Z}_+ \quad \exists \ q\in\mathbb{Z}_+: d(T^n\omega,T^{n+q}\omega)<\varepsilon \text{, to $0\leq n\leq rq$}\}
	\end{equation}
\end{definition}
Considering $\varepsilon(k)=\frac{1}{k}$, $k\in\mathbb{Z}_+$, the set $PRP(T)$ can be expressed as:
	\begin{equation}\label{Kap4Eq3}
		PRP(T)=\{\omega\in\Omega \mid \forall k\in\mathbb{Z}_+ \quad\exists \ q\in\mathbb{Z}_+: d(T^n\omega,T^{n+q}\omega)<\frac{1}{k} \text{ , to $0\leq n\leq kq$}\}
	\end{equation}
\begin{lemma}\label{q infinito}
	Consider $\omega\in PRP(T)$. Given $k$ and $r\in\mathbb{Z}_+$, let $q_k$ be such that $d(T^n\omega,T^{n+q_k}\omega)<\frac{1}{k}$, for $0\leq n\leq rq_k$. Then:
	\begin{equation*}
		\lim_{k\rightarrow\infty}q_k=\infty
	\end{equation*}
\end{lemma}
\begin{proof}
	Let $\omega\in PRP(T)$. Reasoning by contradiction, we assume that there exists $L\in\mathbb{Z}_+$ such that $q_k<L$ for all positive integer $k$. By definition \ref{PRP}, there exist $K\in\mathbb{Z}_+$ such that for $n\in\{0,\ldots, KL\}$, we have
	\begin{equation*}
		d(T^n\omega,T^{n+L}\omega)\geq\frac{1}{K},
	\end{equation*}
	otherwise, $q_K=L$ which does not satisfy $q_k<L$ for every $k\in\mathbb{Z}_+$. On the other hand, as $\{T^n\omega\}_{n\geq0}$ satisfies \textit{RP}, then for every $k$ there exists $q_k\in\{0,1,\ldots L-1\}$ with $d(T^n\omega,T^{n+q_k}\omega)<\frac{1}{k}$, where $0\leq n\leq kq_k$.\\

Let $\varepsilon_0=\min_{q_k}d(T^n\omega,T^{n+q_k}\omega)$ and $m\in\mathbb{Z}_+$ such that $0<\frac{1}{m}<\min\{\frac{1}{K},\varepsilon_0\}$. Therefore, as $\omega\in PRP(T)$, there must be a $q_m\in\mathbb{Z}$ with $d(T^n\omega, T^{n+q_m}\omega)<\frac{1}{m}$ for every $0\leq n\leq mq_m$. However, since $\frac{1}{m}<\varepsilon_0$, then $q_m\notin\{0,1,\ldots L-1\}$. It follows that $q_m\geq L$, but this contradicts $q_k<L$ for all $k$. We conclude that $\lim_{k\rightarrow\infty}q_k=\infty$.
 \end{proof}
 
In the following two propositions we study some properties of the set $PRP(T)$.
\begin{lemma}\label{PRP es G_delta}
	$PRP(T)$ is a $G_\delta$ set (i.e. $PRP(T)$ is the countable intersection of open subsets) on $\Omega$.
\end{lemma}
\begin{proof}
	Let $k$ and $q$ positive integers, consider:
	\begin{equation*}
		A_k(q)=\left\{\omega\in\Omega\mid \max_{0\leq n\leq kq} d(T^n\omega, T^{n+q}\omega)<\frac{1}{k}\right\}
	\end{equation*}
	we have, that $A_k(q)$ is an open subset of $\Omega$. Indeed, given the function $f$:
	\begin{equation*}
		\begin{aligned}
			f:\Omega&\rightarrow\Omega\times\Omega\\
			\omega&\mapsto (T^n\omega, T^{n+q}\omega)
		\end{aligned}
	\end{equation*}
	We observe that $f$ is continuous, since by hypothesis $T:\Omega\rightarrow\Omega$ is a homeomorphism. Furthermore, we know that, the metric $d:\Omega\times \Omega\rightarrow \mathbb{R}^+$ is a continuous function, so that $d\circ f:\Omega\rightarrow\mathbb{R}^+$ is continuous, thus:
	\begin{equation*}
		d\circ f(A_k(q))=[0,\frac{1}{k})
	\end{equation*}
	where $[0,\frac{1}{k})$ is open set on $\mathbb{R}^+$. We conclude that $A_k(q)$ is an open set on $\Omega$.\\

	Now, we will show the following equality
	\begin{equation}\label{Kap4Eq4}
		\bigcap_{k\geq1}\bigcap_{m\geq1}\bigcup_{q\geq m} A_k(q)= PRP(T)
	\end{equation}
	Let $\omega\in \bigcap_{k\geq1}\bigcap_{m\geq1}\bigcup_{q\geq m} A_k(q)$, then for every $k$ and $m$ into $\mathbb{Z}$ there exist $q\in\mathbb{Z}_+$ such that $d(T^n\omega, T^{n+q}\omega)<\frac{1}{k}$, for all $n\in\{0,\ldots,kq\}$. Then by definition \ref{PRP}, we have that $\omega\in PRP(T)$ i. e.,
	\begin{equation*}
\bigcap_{k\geq1}\bigcap_{m\geq1}\bigcup_{q\geq m} A_k(q)\subseteq PRP(T)
	\end{equation*}
 
	In a similar way, if $\omega\in PRP(T)$ then for $k\in\mathbb{Z}_+$ there exists $q$ such that $d(T^n\omega, T^{n+q}\omega)<\frac{1}{k}$, so that $\omega \in A_k(q)$ for all positive integer $k$ hence:
	\begin{equation*}
		\omega\in \bigcap_{k\geq1}\bigcap_{m\geq1}\bigcup_{q\geq m} A_k(q)
	\end{equation*}
	so, $PRP(T)$ is the countable intersection of open sets and therefore it is a $G_\delta$ set.\end{proof}

\begin{lemma}
\label{PRP(T) es invariante}
	$PRP(T)$ is $T-$ invariant, this is $T(PRP(T))\subseteq PRP(T)$.
\end{lemma}
\begin{proof}
	Given $\omega\in PRP(T)$, let us show $T\omega\in PRP(T)$, \textit{i.e.} for all $k\in\mathbb{Z}_+$ there exists $q$ such that $d(T^n(T\omega), T^{n+q}(T\omega))<\frac{1}{k}$, for $0\leq n\leq kq$.\\
	
	If $\omega\in PRP(T)$, then there exists $q\in\mathbb{Z}_+$ such that $d(T^n\omega, T^{n+q}\omega)<\frac{1}{k+1}$, for $0\leq n\leq (k+1)q$. this implies in particular that for $0\leq n\leq kq$:
	\begin{equation*}
		d(T^{n+1}\omega, T^{n+1+q}\omega)<\frac{1}{k+1}\quad\Rightarrow\quad d(T^n(T\omega), T^{n+q}(T\omega))<\frac{1}{k}
	\end{equation*}
	then, $\{T^n(T\omega)\}_{n\geq0}$ satisfies \textit{RP}, so $T\omega\in PRP(T)$.
\end{proof}

Consequently, we have that $\{T^n\omega\}_{n\geq0}\subseteq PRP(T)$ for all $\omega\in PRP(T)$.

\begin{definition}
\label{repeticionTopyMet}
	(BD, \cite{damanikBase} \textit{p.650}). 
 We say that the dynamical system $\{\Omega,T\}$ satisfies:
 
	\begin{itemize}
		\item[\textit{(i)}] \textit{Topological repetition property (TRP)}, if $PRP(T)$ is dense on $\Omega$. 
		\item[\textit{(ii)}] \textit{Metric repetition property (MRP)}, if $\mu(PRP(T))>0$, where $(\Omega, \mathcal{B}, \mu)$ is a measure space.
		\item[\textit{(iii)}] \textit{Global repetition (GRP)}, if $PRP(T)=\Omega$.
	\end{itemize}
\end{definition}

To conclude this section, in the next lemmas we will demonstrate that \textit{TRP} and \textit{MRP} are sufficient conditions to guarantee that the set $PRP(T)$ is generic from the topological and measure-theoretical point of view, respectively, on $\Omega$.

\begin{lemma}
\label{TRP implica residual}
	If $\{\Omega, T\}$ satisfies \textit{TRP}, then the set $PRP(T)$ is residual into $\Omega$.\end{lemma}
\begin{proof}
	By lemma (\ref{PRP es G_delta}) the set {\it PRP(T)} is $G_\delta$. Now, using the hypothesis then $PRP(T)$ is dense, we conclude that $PRP(T)$ is $G_\delta$ and dense, and therefore, by definition, a residual on $\Omega$.
 \end{proof}

\begin{lemma}
\label{transitividad implica TRP}
	Given $\alpha\in PRP(T)$, such that $\{T^n\alpha\}_{n\geq0}$ is dense in $\Omega$. Then:	
	\begin{itemize}
		\item[\textit{(i)}] System $\{\Omega, T\}$ satisfies \textit{TRP}.
		\item[\textit{(ii)}] Given $q\in\mathbb{Z}$, the set $\{T^{q+j}\alpha\}_{j\geq0}$ is dense in $\Omega$.
	\end{itemize}
\end{lemma}
\begin{proof}\textcolor{white}{Sea:}
	
	\textit{(i)}\quad Let $\alpha\in PRP(T)$. By lemma \ref{PRP(T) es invariante}, $\{T^n\alpha\}_{n\geq0}\subseteq PRP(T)$. Therefore, if $\{T^n\alpha\}_{n\geq0}$ is dense in $\Omega$, then \textit{PRP(T)} contains a dense subset of $\Omega$, and consequently \textit{PRP(T)} is dense in $\Omega$.\\
	
	\textit{(ii)}\quad Let $\omega\in\Omega$. Define the set: $A=\{\alpha, T\alpha,\ldots, T^{q-1}\alpha\}$. As $A$ is finite, there exists $\delta_0>0$ such that $B_{\delta_0}(\omega)$, i.e. the open ball centered at $\omega$ with radius $\delta_0$ is disjoint with $A$:
	\begin{equation*}
		B_{\delta_0}(\omega) \cap A=\emptyset
	\end{equation*}
	Given that, by hypothesis, the set $\{T^n\alpha\}_{n\geq0}$ is dense in $\Omega$, then there exists a $j\geq0$ such that $T^j\alpha\in B_{\delta_0}(\omega)$ and therefore:
	\begin{equation*}
		\{T^n\alpha\}_{n\geq0}\cap B_{\delta_0}(\omega)\neq\emptyset
	\end{equation*}
	As $B_{\delta_0}(\omega) \cap A=\emptyset$ and $T^j\alpha\in B_{\delta_0}(\omega)$, then $j\geq q$. We conclude that $\{T^{q+j}\alpha\}_{n\geq0}$ is dense in $\Omega$.\end{proof}

\begin{lemma}
\label{MRP implica medida completa}
	If $\{\Omega, T\}$ satisfies \textit{MRP} and $T$ is $\mu-$ergodic, then $\mu(PRP(T))=\mu(\Omega)$.
\end{lemma}
\begin{proof}
	To simplify calculations, let us denote $PRP(T)=A$. Assume that the system $\{\Omega, T\}$ satisfies \textit{MRP} and $\mu(A)>0$. We will show that $\mu(A)=1$. Let
	\begin{equation*}
		B=\bigcup_{n\geq0}T^{-n}A
	\end{equation*}
	then $T^{-1}B=\bigcup_{n\geq1}T^{-n}A$, therefore $T^{-1}B\subseteq B$. This implies that $\mu(B)=0$ or $\mu(B)=1$, as $T$ is $\mu-$ ergodic. As $A\subseteq B$ and $\mu(A)>0$ by hypotesis, then $\mu(B)>0$, and therefore $\mu(B)=1$.\\
	Furthermore, because $T$ is invariant, $\mu(T^{-1}(B))=1$. We note that:
	\begin{equation*}
		\begin{aligned}
			1=\mu(T^{-1}(B))&=\mu(A)+\mu(T^{-1}(B)\setminus A)
		\end{aligned}
	\end{equation*}
	moreover
	\begin{equation*}
		\begin{aligned}
			1=\mu(B)&=\mu(A)+\mu(B\setminus A)
		\end{aligned}
	\end{equation*}
	then we have that $\mu(B\setminus A)=\mu(T^{-1}(B)\setminus A)$.\\
	
	Now suppose that $\mu(B\setminus A)>0$. Note that $T^{-1}(B\setminus A)\subseteq B\setminus A$, indeed:
	\begin{equation*}
		\begin{aligned}
			x\in T^{-1}(B\setminus A)\Rightarrow T(x)\in B\setminus A& \Rightarrow T(x) \in B \ \text{and} \ T(x)\notin A\\
			&\Rightarrow x \notin A &&\text{lemma \ref{PRP(T) es invariante}}\\
			&\Rightarrow x\in B\setminus A&&\text{because $T^{-1}B\subseteq B$}
		\end{aligned}
	\end{equation*}
	By ergodicity of $T$, we have that $\mu(B\setminus A)=1$, so $\mu(A)=0$, which contradicts the hypothesis $\mu(A)>0$. Therefore $\mu(B\setminus A)=0$, and we conclude that $\mu(A)=1$, \textit{i.e.} $\mu(PRP(T))=1$.
\end{proof}

 Boshernitzan and Damanik demonstrate that \textit{MRP} is sufficient condition for the continuous spectrum to be a generic property, from the measure-theoretical point of view, of the familiy of ergodic Schrödinger operators $\{H_\omega\}_{\omega\in\Omega}$ (\cite{damanikBase} \textit{theorem 2}, p.650). However, this result is beyond the scope of this paper. In the next section we will study Gordon's lemma, result that will allow us to show that the continuous spectrum is a generic property, from the topological point of view, of $\{H_\omega\}_{\omega\in\Omega}$.

\section{Gordon's lemma}
\label{Gordon}
Gordon's lemma is the main analytical tool that supports theorems \textbf{A} and \textbf{B}. In this section we present the definition of Gordon potential and study some properties of matrix representation of Schrödinger operators. We also provide a detailed proof of Gordon's lemma.
\subsection{Gordon potential}
\begin{definition}\label{Potencial de Gordon}
	A bounded function $V:\mathbb{Z}\rightarrow\mathbb{R}$ is a \textit{Gordon potential} if there exists a sequence of periodic functions $\{V_m\}_{m\in\mathbb{Z}_+}: \mathbb{Z}\rightarrow\mathbb{R}$, where $V_m(n)=V_m(n+T_m)$ and $T_m\rightarrow\infty$, that satisfies the following two conditions:
	\begin{itemize}
		\item[\textit{(i)}] $\sup_{n,m}|V_m(n)|<\infty$.
		
		\item[\textit{(ii)}] $\sup_{|n|\leq 2T_m}|V_m(n)-V(n)|\leq Cm^{-T_m}$, for a $C>0$.
	\end{itemize}
\end{definition}

A Gordon potential is essentially a function that, for $|n|\leq 2T_m$, can be approximated by a sequence of periodic functions. In addition to any periodic function, a nontrivial example of a Gordon potential is $V(n)=\sin(an)+\sin(bn)$ where $a$ and $b$ are such that $\frac{a}{b}\notin\mathbb{Q}$. Gordon's lemma, which is stated below, describes the spectral type of the discrete Schrödinger operator in $\ell^2(\mathbb{Z})$:
 \begin{equation}\label{Kap3Eq100}
   H[\psi(n)]=\psi(n+1)+\psi(n-1)+V(n)\psi(n)
 \end{equation}
 whose potential function $V(n)$ satisfies the definition \ref{Potencial de Gordon}.

\begin{theorem}
(Gordon's, lemma) Let $V:\mathbb{Z}\rightarrow\mathbb{R}$ be a Gordon potential. For $E\in\mathbb{C}$, if $\psi$ is a solution to:
	\begin{equation}\label{Kap3Eq9}
		\psi(n+1)+\psi(n-1)+V(n)\psi(n)=E\cdot \psi(n)
	\end{equation}
	Then:
	\begin{equation}\label{Kap3Eq10}
		\limsup_{|n|\rightarrow\infty}\frac{\psi(n+1)^2+\psi(n)^2}{\psi(1)^2+\psi(0)^2}\geq\frac{1}{4}
	\end{equation}
	\end{theorem}
	Equation \ref{Kap3Eq10} implies that $\lim_{|n|\rightarrow\infty}\psi(n)>0$. Therefore:
	\begin{equation}\label{Kap3Eq110}
	  \lVert\psi(n)\rVert^2=\sum_{n\in\mathbb{Z}}|\psi(n)|^2=\infty
	\end{equation}
	as a consequence of equation \ref{Kap3Eq110} we conclude that $\psi(n)\notin\ell^2(\mathbb{Z})$, therefore the operator $H$ in $\ell^2(\mathbb{Z})$ (equation \ref{Kap3Eq100}) has no eigenvalues (\textit{i.e.} its point spectrum is empty). The following theorem provides a characterization of Gordon's potential that will be useful later on the demonstration of Theorems A and B.

\begin{theorem}\label{Equivalencia potencial de Gordon}
A bounded function $V:\mathbb{Z}\rightarrow\mathbb{R}$ is a Gordon potential if and only if there exists a sequence $\{q_m\}_{m\in\mathbb{Z}_+}$ such that $q_m\rightarrow\infty$ and for a $C>0$ and every $m\geq1$:
	\begin{equation*}
		\max_{1\leq n\leq q_m}|V(n)-V(n+q_m)|\leq Cm^{-q_m} \quad \text{and} \quad \max_{1\leq n\leq q_m}|V(n)-V(n-q_m)|\leq Cm^{-q_m}
	\end{equation*}
\end{theorem}
\begin{proof}
``$\Rightarrow$'' Suppose that $V(n)$ is a Gordon potential. Using triangular inequality and the fact that $V_m$ is a periodic function with period $T_m$:
	\begin{equation*}
		\begin{aligned}
			\max_{1\leq n\leq T_m}|V(n)-V(n\pm T_m)|&=\max_{1\leq n\leq T_m}|V(n)-V_m(n)+V_m(n\pm T_m)-V(n\pm T_m)|\\
			&\leq \sup_{1\leq n\leq T_m}|V_m(n\pm T_m)-V(n\pm T_m)|+\sup_{1\leq n\leq T_m}|V_m(n)-V(n)|\\
			&\leq 2Cm^{-T_m}
		\end{aligned}
	\end{equation*}
	by taking $q_m=T_m$ and $\hat{C}=2C$, we conclude that:
	\begin{equation*}
		\max_{1\leq n\leq q_m}|V(n)-V(n\pm q_m)|\leq\hat{C}m^{-qm}
	\end{equation*}
	where $\lim_{m\rightarrow\infty}q_m=\infty$.\\
	
	``$\Leftarrow$'' Suppose that there exists $C>0$ such that $\max_{1\leq n\leq q_m}|V(n)-V(n\pm q_m)|\leq Cm^{-q_m}$ for a sequence of positive integers $\{q_m\}_{m\in\mathbb{Z}_+}$ such that $q_m\rightarrow\infty$. We want to construct, for each $m\in\mathbb{Z}$, a periodic function $V_m:\mathbb{Z}\rightarrow\mathbb{R}$ with period $T_m$, such that $\sup_{n,m}|V_m(n)|<\infty$ and $\sup_{|n|\leq 2T_m}|V_m(n)-V(n)|\leq Cm^{-T_m}$.
 
  \quad
	
	Let $m\in\mathbb{Z}$. By hypothesis, for each $n\in\{1,\ldots, q_m\}$ there exist $r_1(m)$ y $r_2(m)$ such that:
	\begin{equation*}
		\begin{aligned}
			V(n-q_m)+r_1(m)&=V(n), \quad&&|r_1(m)|\leq Cm^{-q_m}\quad\text{y}\\
			V(n)+r_2(m)&=V(n+q_m), \quad&& |r_2(m)|\leq Cm^{-q_m}
		\end{aligned}
	\end{equation*}
  	Initially consider the function $V_m: \{1-q_m,\ldots,2q_m\}\rightarrow\mathbb{R}$ defined as:
	\begin{equation*}
		V_m(n)=
		\begin{cases}
			V(n)+r_1(m) \quad&\text{if $1-q_m\leq n\leq 0$}\\
			V(n) \quad&\text{if $1\leq n\leq q_m$}\\
			V(n)-r_2(m) \quad&\text{if $q_m+1\leq n\leq 2q_m$}
		\end{cases}
	\end{equation*}
	For all $n\in\{1,\ldots,q_m\}$ the function $V_m$ is periodic:
	\begin{equation*}
		\begin{aligned}
			V_m(n+q_m)&=V(n+q_m)-r_2(m) &&\quad\text{definition of $V_m(n)$, because $q_m+1\leq n+q_m\leq 2q_m$}\\
			&=V(n)&&\quad\text{by hypothesis}\\
			&=V_m(n)&&\quad\text{definition of $V_m(n)$}
		\end{aligned}
	\end{equation*}
	Similarly:	
	\begin{equation*}
		\begin{aligned}
			V_m(n-q_m)&=V(n-q_m)+r_1(m) &&\quad\text{definition of $V_m(n)$ because $1-q_m\leq n-q_m\leq 0$}\\
			&=V(n)&&\quad\text{by hypothesis}\\
			&=V_m(n)&&\quad\text{definition of $V_m(n)$}
		\end{aligned}
	\end{equation*}
	therefore $V_m$ is periodic, with period $q_m$.
 
  \quad
  
  Function $V_m$ is extended to $\mathbb{Z}$ by establishing, for each $j\in\mathbb{Z}: V_m(j)=V_m(n)$, where $n\in\{1,\ldots q_m\}$ and $j\equiv_{q_m}n$. We conclude that $V_m: \mathbb{Z}\rightarrow\mathbb{Z}$ is a periodic function with period $q_m$.
	
	\quad
	
	Additionally, since the function $V:\mathbb{Z}\rightarrow \mathbb{R}$ is bounded, there exists $K$ such that $V(n)\leq K$ for all $n\in\mathbb{Z}$. Let $r(m)=\max\{r_1(m), r_2(m)\}$, note that for all $n$ and $m\in\mathbb{Z}$:
	\begin{equation*}
		\begin{aligned}
			|V_m(n)|\leq |V(n)|+|r(m)|\leq K+Cm^{-q_m}<\infty
		\end{aligned}
	\end{equation*}
	then $\sup_{n,m}|V_m(n)|<\infty$. Finally:
	\begin{equation*}
		|V_m(n)-V(n)|\leq |r(m)|\leq Cm^{-q_m} \Rightarrow \sup_{|n|\leq 2q_m}|V_m(n)-V(n)|\leq Cm^{-q_m}
	\end{equation*}
	Taking $q_m=T_m$, it is concluded that the sequence of functions $\{V_m\}_{m\in\mathbb{Z}^+}$ satisfies definition \ref{Potencial de Gordon} and therefore $V(n)$ is a Gordon potential.
\end{proof}

\subsection{Matrix representation of Schrödinger operators}
Let $V(n)$ be a Gordon potential and $\Psi(n)$ the column vector $(\psi(n), \psi(n+1))$ where $\psi(n)$ is a solution to the equation:
\begin{equation}\label{Kap3Eq2}
	\psi(n+1)+\psi(n-1)+V(n)\psi(n)=E\psi(n)
\end{equation}
which means that $E\in\mathbb{C}$ is an eigenvalue of $H$ (equation \ref{Kap3Eq100}) with given initial condition $\Psi(0)$. For $n>0$, equation \eqref{Kap3Eq2} can be written in matrix form:
\begin{equation}\label{Kap3Eq3}
	\Psi(n)=A(n)\cdots A(1)\Psi(0), \quad\text{where}\quad A(n)=\begin{pmatrix}
		0&1\\-1&E-V(n)
	\end{pmatrix}
\end{equation}
Similarly, let $\Psi_m(n)=(\psi_m(n), \psi_m(n+1))$ and consider the equation:
\begin{equation*}
	\psi_m(n-1)+\psi_m(n+1)+V_m(n)\psi_m(n)=E\psi_m(n)
\end{equation*}
with initial condition $\Psi_m(0)=\Psi(0)$. Therefore:
\begin{equation}\label{Kap3Eq4}
	\Psi_m(n)=A_m(n)\cdots A_m(1)\Psi(0), \quad\text{where}\quad A_m(n)=\begin{pmatrix}
		0&1\\-1&E-V_m(n)
	\end{pmatrix}
\end{equation}
In the following two lemmas we study some properties of $A_m(n)$ and $A(n)$.

\begin{theorem}\label{Aux1}
Let $A(n)$ and $A_m(n)$ be according to the equations \eqref{Kap3Eq3} and \eqref{Kap3Eq4}. Then:
	\begin{equation*}
		\lVert A_m(n)\cdots A_m(1)-A(n)\cdots A(1)\rVert\leq n\cdot[\sup_{m,j}\lVert A_m(j)\rVert]^{n-1}\cdot[\sup_{1\leq j\leq n}\lVert A_m(j)-A(j)\rVert]
	\end{equation*}
\end{theorem}
\begin{proof}
Matrix $A_m(n)\cdots A_m(1)-A(n)\cdots A(1)$ can be written as a telescopic sum:
	\begin{equation}\label{Kap3Eq5}
		\begin{aligned}
			&A_m(n)\cdots A_m(1)-A(n)\cdots A(1)\\&=[A_m(n)-A(n)]\cdot[A_m(n-1)\cdots A_m(1)]\\
			&+[A(n)]\cdot[A_m(n-1)-A(n-1)]\cdot[A_m(n-2)\cdots A_m(1)]+\cdots\\
			&+[A(n)\cdots A(n-j+1)]\cdot[A_m(n-j)-A(n-j)]\cdot[A_m(n-j-1)\cdots A_m(1)]+\cdots\\
			&+[A(n)\cdots A(2)]\cdot[A_m(1)-A(1)]
		\end{aligned}
	\end{equation}
	Since $\lVert AB\rVert\leq\lVert A\rVert\cdot\lVert B\rVert=\lVert B\rVert\cdot\lVert A\rVert$, the norm of each summand on the right-hand side of the equation \eqref{Kap3Eq5} can be bounded as follows, for all $0\leq j\leq n-1$:
	\begin{equation}\label{Kap3Eq6}
		\begin{aligned}
			&\lVert [A(n)\cdots A(n-j+1)]\cdot[A_m(n-j)-A(n-j)]\cdot[A_m(n-j-1)\cdots A_m(1)]\rVert\\
			&\leq \lVert [A(n)\cdots A(n-j+1)]\rVert \cdot\lVert[A_m(n-j)-A(n-j)\rVert\cdot\lVert[A_m(n-j-1)\cdots A_m(1)]\rVert\\
			&=\lVert [A(n)\cdots A(n-j+1)]\rVert\cdot\lVert[A_m(n-j-1)\cdots A_m(1)]\rVert\lVert[A_m(n-j)-A(n-j)]\rVert\\
			&\leq
			[\sup_{m,j}\lVert A_m(j)\rVert]^{n-1}\cdot[\sup_{1\leq j\leq n}\lVert A_m(j)-A(j)\rVert]
		\end{aligned}
	\end{equation}
	In the above equation:
	\begin{equation*}
		\sup_{m,j}\lVert A_m(j)\rVert=\sup_{1\leq j\leq n}\{\lVert A_m(j)\rVert, \lVert A(j)\rVert\}
	\end{equation*}
	From the equations \eqref{Kap3Eq5} and \eqref{Kap3Eq6} we obtain:
	\begin{equation*}
		\lVert A_m(n)\cdots A_m(1)-A(n)\cdots A(1)\rVert\leq n[\sup_{m,j}\lVert A_m(j)\rVert]^{n-1}[\sup_{1\leq j\leq n}\lVert A_m(j)-A(j)\rVert]
\end{equation*}
\end{proof}
\begin{theorem}\label{Aux2}
Let $x$ be a vector such that $\lVert x\rVert=1$ and $B$ an invertible $2\times2$ matrix. Then:
	\begin{equation*}
		\max_{a=\pm1,\pm2}\lVert B^ax\rVert\geq\frac{1}{2}
	\end{equation*}
	In particular, for $\Psi_m(n)=A_m(n)\cdots A_m(1)\Psi(0)$, we have that:
	\begin{equation*}
		\max_{a=\pm1,\pm2}\lVert \Psi_m(aT_m)\rVert\geq\frac{1}{2}\lVert\Psi(0)\rVert
	\end{equation*}	
\end{theorem}
\begin{proof}
Cayley-Hamilton theorem states that if $q(\lambda)=\sum_{k=0}^{n}a_k\lambda^k$ is the characteristic polynomial of a linear transformation $T$ in a vector space $V$ of dimension $n$, then $q(A)=0$, where $A$ is the square $n\times n$ matrix associated to the operator $T$ (Axler, \citeyear{axler2}. \textit{theorem 8.37}).
	Therefore, if $q(\lambda)=a_2\lambda^2+a_1\lambda+a_0$ is the characteristic polynomial associated with the matrix $B_{2\times2}$, then:
	\begin{equation}\label{Kap3Eq7}
		a_2B^2+a_1B+a_0=0
	\end{equation}
	To demonstrate theorem \ref{Aux2} we consider three cases. First $|a_2|=\max_{0\leq i\leq 2}|a_i|$. Multiplying right-hand side of the equation \eqref{Kap3Eq7} by $\frac{1}{a_2}B^{-2}x$, and making $c_1=\frac{a_1}{a_2}$ y $c_0=\frac{a_0}{a_2}$, we obtain:
	\begin{equation*}
		x+c_1B^{-1}x+c_0B^{-2}x=0
	\end{equation*}
	where $|c_1|, |c_0|\leq1$. Taking the norm of the above expression:
	\begin{equation*}
		\begin{aligned}
			\lVert x\rVert=\lVert c_1B^{-1}x+c_0B^{-2}x\rVert \quad&\Rightarrow\quad
			1\leq |c_1|\cdot\lVert B^{-1}x\rVert+|c_0|\cdot\lVert B^{-2}x\rVert\\
			&\Rightarrow\quad\max\{\lVert B^{-1}x\rVert, \lVert B^{-2}x\rVert\}\geq\frac{1}{2}
		\end{aligned}
	\end{equation*}
	The second case is $|a_1|=\max_{0\leq i\leq 2}|a_i|$. Multiplying equation \eqref{Kap3Eq7} by $\frac{1}{a_1}B^{-1}x$:
	\begin{equation*}
		d_2Bx+x+d_0B^{-1}x=0
	\end{equation*}
	where $|d_2|=\frac{a_2}{a_1}\leq1$ y $|d_0|=\frac{a_0}{a_1}\leq1$. Then:
	\begin{equation*}
		\begin{aligned}
			\lVert x\rVert=\lVert d_2Bx+d_0B^{-1}x\rVert \quad&\Rightarrow\quad
			1\leq |d_2|\cdot\lVert Bx\rVert+|d_0|\cdot\lVert B^{-1}x\rVert\\
			\quad&\Rightarrow\quad\max\{\lVert Bx\rVert, \lVert B^{-1}x\rVert\}\geq\frac{1}{2}
		\end{aligned}
	\end{equation*}
	The third case is $|a_0|=\max_{0\leq i\leq 2}|a_i|$. An analogous reasoning to the one just presented, multiplying now the equation \eqref{Kap3Eq7} by $\frac{1}{a_0}x$ allows us to affirm that:
	\begin{equation*}
		\max\{\lVert B^2x\rVert, \lVert Bx\rVert\}\geq\frac{1}{2}
	\end{equation*}
	It is concluded that invertible $2\times 2$ matrix $B$ satisfies inequality:
	\begin{equation}\label{Kap3Eq8}
		\max_{a=\pm1,\pm2}\lVert B^ax\rVert\geq\frac{1}{2}
	\end{equation}
	Let $B_m(n)=A_m(n)\cdots A_m(1)=\Pi_{j=1}^{n}A_m(j)$.
 
 Note that $A_m(n)=A_m(n+T_m)$, since $V_m(n)=V_m(n+T_m)$. Then:
	\begin{equation*}
		\begin{aligned}
			{B_m^a(T_m)}=(\Pi_{j=1}^{T_m}A_m(j))^a&=\Pi_{j=1}^{T_m}A_m(j)\cdots\Pi_{j=1}^{T_m}A_m(j)\\
			&=\Pi_{j=1}^{T_m}A_m(j)\cdot \Pi_{j=T_m+1}^{2T_m}A_m(j)\cdots\Pi_{j=(a-1)(T_m+1)}^{aT_m}A_m(j)\\
			&=\Pi_{j=1}^{aT_m}A_m(j)\\
			&=B_m(aT_m)
		\end{aligned}
	\end{equation*}
	Replacing $x$ by $\frac{\Psi(0)}{\lVert\Psi(0)\rVert}$ and $B$ by $B_m(T_m)$ in the equation \eqref{Kap3Eq8}:
	\begin{equation*}
		\begin{aligned}
			\max_{a=\pm1,\pm2}\lVert B_m^a(T_m)\Psi(0)\rVert = \max_{a=\pm1,\pm2}\lVert B_m(aT_m)\Psi(0)\rVert\geq\frac{1}{2}\lVert\Psi(0)\rVert
		\end{aligned}\end{equation*}
\end{proof}
 

\subsection{Proof of Gordon's lemma}
\label{demostración lema de Gordon}
\begin{proof}
This proof is based on Simon's suggestions (\cite{simon} theorem 7.1, p.476). On the one hand, note that:
	\begin{equation}\label{Kap3Eq11}
		\begin{aligned}
			\lVert \Psi_m(n)-\Psi(n)\rVert&=\lVert [A_m(n)\cdots A_m(1)-A(n)\cdots A(1)]\cdot\Psi(0)\rVert\\
			&\leq \lVert (A_m(n)\cdots A_m(1)-A(n)\cdots A(1))\rVert\cdot\lVert\Psi(0)\rVert\\
			&\leq n\cdot[\sup_{m,j}\lVert A_m(j)\rVert]^{n-1}\cdot[\sup_{1\leq j\leq n}\lVert A_m(j)-A(j)\rVert]\cdot \lVert\Psi(0)\rVert
		\end{aligned}
	\end{equation}
	where the last inequality is justified in the theorem \ref{Aux1}. On the other hand:
	\begin{equation*}
		A_m(j)-A(j)=\begin{pmatrix}
			0&0\\0&V(n)-V_m(n)
		\end{pmatrix}\quad\Rightarrow\quad\lVert A_m(j)-A(j)\rVert\leq|V(n)-V_m(n)|
	\end{equation*}
	By assumption $V(n)$ is a Gordon potential. Therefore:
	\begin{equation}\label{Kap3Eq12}
		\sup_{|n|\leq 2T_m}|V_m(n)-V(n)|\leq Cm^{-T_m}\quad\Rightarrow\quad\lim_{m\rightarrow\infty}\sup_{|n|\leq 2T_m}\lVert A_m(j)-A(j)\rVert=0
	\end{equation}
	From the equations \eqref{Kap3Eq11} and \eqref{Kap3Eq12} we obtain:
	\begin{equation*}
		\sup_{|n|\leq 2T_m} \lVert \Psi_m(n)-\Psi(n)\rVert\rightarrow0, \quad\text{as }m\rightarrow\infty
	\end{equation*}
	then in particular:
	\begin{equation*}
		\max_{a=\pm1,\pm2}\lVert\Psi(aT_m)-\Psi_m(aT_m)\rVert\rightarrow0, \quad\text{as }m\rightarrow\infty
	\end{equation*}
	As $\max_{a=\pm1,\pm2}\lVert \Psi_m(aT_m)\rVert\geq\frac{1}{2}\lVert\Psi(0)\rVert$ (theorem \ref{Aux2}), it follows from the above equation that:
	\begin{equation}\label{Kap3Eq13}
		\max_{a=\pm1,\pm2}\lVert \Psi(aT_m)\rVert\geq\frac{1}{2}\lVert \Psi(0)\rVert
	\end{equation}
	Consequently:
	\begin{equation*}
		\begin{aligned}
			\limsup_{|n|\rightarrow\infty}\lVert \Psi(n)\rVert\geq \max_{a=\pm1,\pm2}\lVert \Psi(aT_m)\rVert\quad&\Rightarrow\quad\limsup_{|n|\rightarrow\infty}\lVert \Psi(n)\rVert\geq\frac{1}{2}\lVert \Psi(0)\rVert &&\quad\text{by equation \eqref{Kap3Eq13}}\\&\Rightarrow\quad\limsup_{|n|\rightarrow\infty} \frac{\lVert \Psi(n)\rVert^2}{\lVert \Psi(0)\rVert^2}\geq \frac{1}{4} &&\quad\text{since $\lVert \Psi(0)\rVert>0$}\\
			&\Rightarrow \quad\frac{\psi(n+1)^2+\psi(n)^2}{\psi(1)^2+\psi(0)^2}\geq\frac{1}{4} &&\quad\text{definition of $\lVert \Psi(n)\rVert$}
		\end{aligned}
	\end{equation*}
	
 
 Therefore: 
	\begin{equation}\label{Kap3Eq14}
		\limsup_{|n|\rightarrow\infty}\frac{\psi(n+1)^2+\psi(n)^2}{\psi(1)^2+\psi(0)^2}\geq\frac{1}{4}
	\end{equation}
\end{proof}

\section{Proof of theorem A}
\label{Proof of theorem A}

\begin{proof}
	We will show that $f(T^n\omega)$ is a Gordon potential for all $f$ in a residual set $\mathcal{F}$ of $C(\Omega)$ and $\omega$ in a residual set $\Omega_f$ of $\Omega$. Recall that the real function $V(n)=f(T^n\omega)$ is a Gordon potential if there exists a sequence of positive integers $q_m\rightarrow\infty$ and a $C>0$ such that (theorem \ref{Equivalencia potencial de Gordon}):
	\begin{equation}\label{Kap4Eq7}
		\begin{aligned}
			\max_{1\leq n\leq q_m}|f(T^n\omega)-f(T^{n+q_m}\omega)|&\leq Cm^{-q_m} \quad\text{and:}\\
			\max_{1\leq n\leq q_m}|f(T^n\omega)-f(T^{n-q_m}\omega)|&\leq Cm^{-q_m}
		\end{aligned}
	\end{equation}
	Demonstration is divided into three steps. Firstly, we construct a residual set $\mathcal{F}\subseteq C(\Omega)$. Secondly, for each $f\in\mathcal{F}$ we construct a residual set $\Omega_f\subset\Omega$. Finally, we show that for every $f\in\mathcal{F}$ and $\omega\in\Omega_f$ the function $f(T^n\omega)$ satisfies both inequalities of the equation \eqref{Kap4Eq7} and therefore is a Gordon potential.\\
	
	\textit{(1) Construction of a residual set $\mathcal{F}\subseteq C(\Omega)$:}
	
	Let $\alpha\in PRP(T)$ \textit{i.e.} $\{T^j\alpha\}_{j\geq0}$ satisfies \textit{RP} and thus, for each $k\in\mathbb{Z}_+$ exists a $q_k$ such that:
	\begin{equation*}
		d(T^j\alpha, T^{j+q_k}\alpha)<\frac{1}{k},\quad\text{for $0\leq j\leq 3q_k$}
	\end{equation*}
	where $\lim_{k\rightarrow\infty}q_k=\infty$ (lemma \ref{q infinito}).\\
	
	For each $k\in\mathbb{Z}_+$ let $B_k=B_k(\alpha,r(k))$ be the open ball centered at $\alpha$ with radius $r(k)$. As $T^j$ is a homeomorphism on $\Omega$ then $\{T^j(B_k)\}_{j=1}^{4q_k}$ is a sequence of open sets in $\Omega$. We represent the union of elements of this sequence as:
		\begin{equation*}
			\begin{aligned}
				&\bigcup_{j=1}^{q_k}T^j(B_k) \cup T^{q_k+j}(B_k) \cup T^{2q_k+j}(B_k) \cup T^{3q_k+j}(B_k)=\bigcup_{j=1}^{q_k}\bigcup_{l=0}^{3}T^{j+lq_k}(B_k)
			\end{aligned}
	\end{equation*}
	Radius $r(k)$ of the open ball $B_k$ centered at $\alpha$ is considered sufficiently small for the following two conditions to be met:
	\begin{itemize}
		\item[\textbf{(a)}] $\overline{T^i(B_k)}\cap \overline{T^j(B_k)}=\emptyset,\quad\forall i,j\in \{1,\ldots,4q_k\}$\quad for $i\neq j$.
		
		\item[\textbf{(b)}] For each $1\leq j\leq q_k$ we have that $\bigcup_{l=0}^{3}T^{j+lq_k}(B_k)$ is contained in a ball of radius $\frac{4}{k}$.
	\end{itemize}
	Imposing the condition \textbf{(a)} is possible since by hypothesis $T$ is a homeomorphism in $\Omega$ and also we are considering a finite number of iterates of $T$. Regarding the condition \textbf{(b)}, note that as $\{T^j\alpha\}_{j\geq0}$ satisfies \textit{RP} then for each $k\in\mathbb{Z}$ we have that:
 
	\begin{equation*}
		\begin{aligned}[c]
				d(T^j\alpha,T^{j+q_k}\alpha)&<\frac{1}{k}\\
				d(T^{j+q_k}\alpha,T^{j+2q_k}\alpha)&<\frac{1}{k}\\
				d(T^{j+2q_k}\alpha,T^{j+3q_k}\alpha)&<\frac{1}{k}
		\end{aligned}
			\qquad\Rightarrow\qquad
		\begin{aligned}[c]
			d(T^j\alpha,T^{j+lq_k}\alpha)&<\frac{3}{k} \quad\text{for $0\leq l\leq 3$}
		\end{aligned}
	\end{equation*}

\begin{figure}[h]
  \centering
  \includegraphics[width=0.6\linewidth]{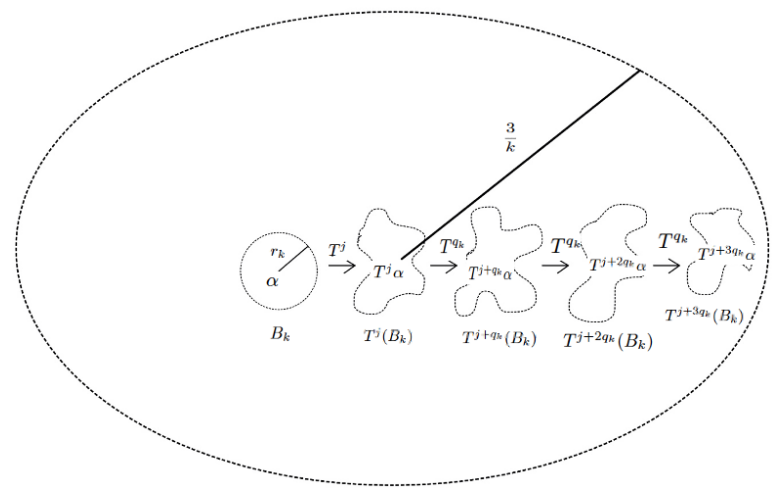}
  \caption{{\small Illustration of conditions \textbf{(a)} and \textbf{(b)}.}}
  \label{fig:DemoTeo1}
\end{figure}

  The utility of conditions \textbf{(a)} and \textbf{(b)} is that they guarantee that the set $\mathcal{F}$ constructed below is residual in $C(\Omega)$. Initially, consider the set:
	\begin{equation*}
		\mathcal{C}_k=\{g\in C(\Omega)\mid \text{$g$ is constant on $\bigcup_{l=0}^3T^{j+lq_k}(B_k)$ for each $j: 1\leq j\leq q_k$}\}
	\end{equation*}
	We define the set $\mathcal{F}_k$ as the open neighborhoods of radio $k^{-q_k}$ centered at $g\in\mathcal{C}_k$:
	\begin{equation}\label{Kap4Eq8}
		\mathcal{F}_k=\{f\in C(\Omega)\mid |f-g|=\sup_{x\in \Omega}|f(x)-g(x)|<k^{-q_k}, \quad\text{para}\quad g\in \mathcal{C}_k\}
	\end{equation}
	For each $m\in\mathbb{Z}_+$ consider:
	\begin{equation*}
		\hat{\mathcal{F}_m}=\bigcup_{k\geq m}\mathcal{F}_k
	\end{equation*}
	The set $\hat{\mathcal{F}_m}$ is open, because it is the countable union of open sets in $C(\Omega)$.
	Additionally $\hat{\mathcal{F}_m}$ is dense in $C(\Omega)$: consider $h\in C(\Omega)$, we will construct a $g\in\hat{\mathcal{F}_m}$ arbitrarily close to $h$. As $\Omega$ is compact, then $h$ is uniformly continuous. Let $x,y\in \bigcup_{l=0}^3T^{j+lq_k}(B_k)$, In accordance with condition \textbf{(b)}: $d(x,y)<4/k$. This implies, due to the uniform continuity of $h$, that $|h(x)-h(y)|\rightarrow0$ for $k\in\mathbb{Z}$ arbitrarily large. For a given $\hat{x}\in \bigcup_{l=0}^3T^{j+lq_k}(B_k)$, define the function:
	\begin{equation}\label{Kap4Eq9}
		g(x)=
		\begin{cases}
			h(\hat{x}) & \text{ if}\quad x\in \bigcup_{l=0}^3T^{j+lq_k}(B_k)\\
			h(x) & \text{ if} \quad x \notin \bigcup_{l=0}^3T^{j+lq_k}(B_k)
		\end{cases}
	\end{equation}
	It follows that $g\in \mathcal{C}_k$ and $h\in \hat{\mathcal{F}_m}$ for every $m\geq1$ therefore $\hat{\mathcal{F}_m}$ is dense in $C(\Omega)$. We conclude that for each $m\in\mathbb{Z}_+$, the set $\hat{\mathcal{F}_m}$ is open and dense in $C(\Omega)$. Therefore:
	\begin{equation*}
		\begin{aligned}
			\mathcal{F}&=\bigcap_{m\geq1}\hat{\mathcal{F}_m}=\bigcap_{m\geq1}\bigcup_{k\geq m}\mathcal{F}_k 
		\end{aligned}
	\end{equation*}
	is a residual set in $C(\Omega)$ since it is the countable intersection of open and dense sets in $C(\Omega)$.
	
	\quad

	\textit{(2) Construction of a residual set $\Omega_f\subseteq\Omega$:}
	
	Let $f\in\mathcal{F}$. By definition of the set $\mathcal{F}$ in the previous step of this proof, there exists a sequence of positive integers $k_l$, with $\lim_{l\rightarrow\infty}k_l=\infty$, such that $f$ belongs to each $\hat{\mathcal{F}_{k_l}}$ (equation \ref{Kap4Eq8}). Moreover, for each $k_l$ there exists a $q_{k_l}$ such that:
	\begin{equation*}
		d(T^j\alpha, T^{j+q_{k_l}}\alpha)<\frac{1}{k_l}, \quad\text{para}\quad1\leq j\leq k_l\cdot q_{k_l}
	\end{equation*}
	where $\lim_{l\rightarrow\infty}q_{k_l}=\infty$, by lemma \ref{q infinito}. According to lemma \ref{transitividad implica TRP} the set $\{T^{q+j}\alpha\}_{j\geq1}$ is dense in $\Omega$ for every $q\in\mathbb{Z}$, then in particular $\{T^{q_{k_l}+j}\alpha\}_{j\geq1}$ is dense in $\Omega$ for each $q_{k_l}$. The union of elements of the latter sequence is represented, for each $m\geq1$, as:
	\begin{equation*}
		\bigcup_{l\geq m}\bigcup_{j=1}^{q_{k_l}}T^{q_{k_l}+j}(\alpha), \quad\text{where:}\quad\bigcup_{l\geq m}\bigcup_{j=1}^{q_{k_l}}T^{q_{k_l}+j}(\alpha)\subseteq \bigcup_{l\geq m}\bigcup_{j=1}^{q_{k_l}}T^{q_{k_l}+j}(B_{k_l})
	\end{equation*}

	The above equation is justified by the fact that $\alpha\in B_{k_l}$, where $B_{k_l}$ is an open ball, in accordance with step \textit{(1)} of this proof. Therefore, the set:
	\begin{equation*}
		\Omega_{f,m}=\bigcup_{l\geq m}\bigcup_{j=1}^{q_{k_l}}T^{q_{k_l}+j}(B_{k_l})
	\end{equation*}
	is dense in $\Omega$ as it contains $\{T^{q_{k_l}+j}\alpha\}_{j\geq1}$, which is dense in $\Omega$. We conclude that:
	\begin{equation*}
		\Omega_f=\bigcap_{m\geq1} \Omega_{f,m}=\bigcap_{m\geq1}\bigcup_{l\geq m}\bigcup_{j=1}^{q_{k_l}}T^{j+q_{k_l}}(B_{k_l})
	\end{equation*}
	is the countable intersection of open and dense sets, and therefore residual subset of $\Omega$.
	
	\quad
	
	\textit{(3) We show that $f(T^n\omega)$ is a Gordon potential, for every $f\in\mathcal{F}$ and every $\omega\in\Omega_f$:}
	
	Let $f\in\mathcal{F}$ and $\omega\in\Omega_f$. Since $\omega\in\Omega_f$, there exists a sequence $k_l\rightarrow\infty$ such that:
	\begin{equation*}
		\omega\in\bigcup_{j=1}^{q_{k_l}}T^{j+q_{k_l}}(B_{k_l})
	\end{equation*}
	Therefore, for each ${k_l}$ exists a $\hat{j}$ ($1\leq\hat{j}\leq q_{k_l}$) such that $\omega\in T^{\hat{j}+q_{k_l}}(B_{k_l})$. This implies that for each $j$, with $1\leq j\leq q_{k_l}$:
	\begin{equation*}
		T^j\omega\in T^{\hat{j}+j+q_{k_l}}(B_{k_l}),\quad 		T^{j+q_{k_l}}\omega\in T^{\hat{j}+j+2q_{k_l}}(B_{k_l})\quad\text{and}\quad 		T^{j-q_{k_l}}\omega\in T^{\hat{j}+j}(B_{k_l})
	\end{equation*}
	Let $\bar{j}=\hat{j}+j$. Therefore:
	\begin{equation*}
		T^j\omega, \quad T^{j+q_{k_l}}\omega \quad\text{and}\quad T^{j-q_{k_l}}\omega\quad\in\bigcup_{1\leq \bar{j}\leq q_{k_l}}\bigcup_{l=0}^{3}T^{\bar{j}+lq_{k_l}}(B_{k_l}), \quad \forall j:\quad1\leq j\leq q_{k_l}
	\end{equation*}
	Consider a function $g\in\mathcal{C}_{k_l}$ . This means, by the definition of the set $\mathcal{C}_{k_l}$ in the previous step of this proof, that $g$ is constant on $\bigcup_{1\leq \bar{j}\leq q_{k_l}}\bigcup_{l=0}^{3}T^{\bar{j}+lq_{k_l}}(B_{k_l})$. In particular:
	\begin{equation*}
		g(T^{j}\omega)=g(T^{j+q_{k_l}}\omega)\quad\text{and}\quad g(T^{j}\omega)=g(T^{j-q_{k_l}}\omega)
	\end{equation*}
	Therefore, if $f\in\mathcal{F}$ then:
	\begin{equation*}
		\begin{aligned}
			|f(T^j\omega)-f(T^{j+q_{k_l}}\omega)|&=|f(T^j\omega)-g(T^j\omega)+g(T^{j+q_{k_l}}\omega)-f(T^{j+q_{k_l}}\omega)|&&\text{as $g(T^j\omega)=g(T^{j+q_{{k_l}}}\omega)$}\\
			&\leq |f(T^j\omega)-g(T^j\omega)|+|f(T^{j+q_{k_l}}\omega)-g(T^{j+q_{k_l}}\omega)|&&\text{by triangular inequality}\\
			&< {k_l}^{-q_{k_l}}+{k_l}^{-q_{k_l}}=2{k_l}^{-q_{k_l}}&&\text{as $f\in\mathcal{F}$}
		\end{aligned}
	\end{equation*}
	Similarly:
	\begin{equation*}
		\begin{aligned}
			|f(T^j\omega)-f(T^{j-q_{k_l}}\omega)|&=|f(T^j\omega)-g(T^j\omega)+g(T^{j-q_{k_l}}\omega)-f(T^{j-q_{k_l}}\omega)|&&\text{as $g(T^j\omega)=g(T^{j-q_{k_l}}\omega)$}\\
			&\leq |f(T^j\omega)-g(T^j\omega)|+|f(T^{j-q_{k_l}}\omega)-g(T^{j-q_{k_l}}\omega)|&&\text{by triangular inequality}\\
			&< {k_l}^{-q_{k_l}}+{k_l}^{-q_{k_l}}=2{k_l}^{-q_{k_l}}&&\text{as $f\in\mathcal{F}$}
		\end{aligned}
	\end{equation*}
	Consequently:
	\begin{equation*}
		|f(T^j\omega)-f(T^{j+q_{k_l}}\omega)|<2{k_l}^{-q_{k_l}} \quad\text{and}\quad|f(T^j\omega)-f(T^{j-q_{k_l}}\omega)|<2{k_l}^{-q_{k_l}}
	\end{equation*}
	
	The above inequality is satisfied for each $j\in\{1,\ldots,q_{k_l}\}$. This leads us to conclude that:
	\begin{equation*}
		\max_{1\leq j\leq q_{k_l}}|f(T^j\omega)-f(T^{j+q_{k_l}}\omega)|<2{k_l}^{-q_{k_l}}\quad\text{and}\quad\quad \max_{1\leq j\leq q_{k_l}}|f(T^j\omega)-f(T^{j-q_{k_l}}\omega)|<2{k_l}^{-q_{k_l}}
	\end{equation*}
	Taking $C=2$ in equation \eqref{Kap4Eq7} we conclude that $f(T^n\omega)$ is a Gordon potential.

 \end{proof}

\section{Proof of theorem B}\label{Proof of theorem B}

\begin{proof}
	We focus on the construction of the residual $\Omega_f$ set in $\Omega$, since the rest of the proof is analogous to that of \textbf{Theorem A}. Let $f\in\mathcal{F}$. This implies that for a subsequence of positive integers $k_l\rightarrow\infty$, $f$ belongs to each set $\mathcal{F}_{k_l}$ (equation \ref{Kap4Eq8}). Moreover, since by hypothesis the system $\{\Omega, T\}$ satisfies \textit{TRP}, we have that $\Omega$ can be covered as follows:
 
	\begin{equation*}
		\Omega\subseteq \bigcup_{\omega\in PRP(T)}B_{\varepsilon_\omega}\omega
	\end{equation*}
	where $B_{\varepsilon_\omega}$ are open balls centered at $\omega\in PRP(T)$ with radius $\varepsilon_\omega>0$. For each $m\in\mathbb{Z}_+$ we define the set:
	\begin{equation*}
		\Omega_{f,m}=\bigcup_{\omega\in PRP(T)}\bigcup_{l\geq m}\bigcup_{j=1}^{q_{k_l}}T^{j+q_{k_l}}(B_{r(k_l)}(\omega))
	\end{equation*}
  For each $m\in\mathbb{Z}_+$, the set $\Omega_{f,m}$ is open since it is the union of open sets. Additionally, $\Omega_{f,m}$ is dense in $\Omega$, since the system $\{\Omega, T\}$ satisfies \textit{TRP} and $T$ is a homeomorphism. Consequently: 
	\begin{equation*}
		\Omega_f=\bigcap_{m\geq 1}\Omega_{f,m}=\bigcap_{m\geq 1}\bigcup_{\omega\in PRP(T)}
\bigcup_{l\geq m}\bigcup_{j=1}^{q_{k_l}}T^{j+q_{k_l}}(B_{r(k_l)}(\omega))	\end{equation*}
  is a residual subset of $\Omega$ because it is the intersection of dense open sets. Lastly, for each $f\in\mathcal{F}$ and $\omega\in\Omega_{f,m}$, the function $f(T^n\omega)$ is a Gordon potential. The argument is the same as presented in the third step of the proof of \textbf{Theorem A} (section \ref{Proof of theorem A}).
  
  \end{proof}

\section{Applications}
\label{section4}
As application of theorems \textbf{A} and \textbf{B}, in this section we study the spectrum of quasi-periodic Schrödinger operators and operators whose potential function is determined by skew-shift on the torus. First, the following theorem allows us to determine the spectral type of a broader class of dynamical systems. 

\begin{theorem}
 \label{thc} 
 Let $(\Omega,d)$ be a compact metric space. Suppose that $T:\Omega\rightarrow\Omega$ satisfies the following two conditions:

\begin{itemize}
		\item[\textit{(i)}] $T$ is minimal: $\overline{\{T^n\omega\}}_{n\geq0}=\Omega, \quad \forall \omega\in\Omega$.
		\item[\textit{(ii)}] $T$ is an isometry: $d(\omega_1,\omega_2)=d(T\omega_1,T\omega_2), \quad \forall \omega_1,\omega_2\in\Omega$.
	\end{itemize}
	 Then the system $\{\Omega, T\}$ satisfies \textit{TRP}. Therefore, by \textbf{Theorem B}, for all $\omega$ in a residual subset of $\Omega$ and $f$ in a residual subset of $C(\Omega)$, the operator:
	\begin{equation}\label{Kap5Eq3}
		\begin{aligned}
			H_\omega: \ell^2(\mathbb{Z})&\rightarrow \ell^2(\mathbb{Z})\\
			\psi(n)&\mapsto \psi(n+1)+\psi(n-1)+f(T^n\omega)\psi(n)
		\end{aligned}
	\end{equation}
has purely continuous spectrum.
\end{theorem}

\begin{proof}

Consider $\omega\in \Omega$. Our objective is to show that $\omega$ belongs to the \textit{PRP(T)} 
\textit{i.e.} that for all $k\in\mathbb{Z}_+$ there exists a $q_k\in\mathbb{Z}_+$ such that $d(T^n\omega, T^{n+q_k}\omega)<\frac{1}{k}$ for $n\in\{0,\ldots, kq_k\}$. By hypothesis $T$ is minimal, then for $k\in\mathbb{Z}_+$ exists a $\hat{q}_k\in\mathbb{Z}_+$ such that $d(\omega, T^{\hat{q}_k}\omega)<\frac{1}{k}$, for all $n\in\{0,\ldots, k\hat{q}_k\}$. Since $T$ is an isometry: $d(\omega, T^{q_k}\omega)=d(T^n\omega, T^{n+q_k}\omega)$. Therefore:
	\begin{equation*}
		d(\omega, T^{\hat{q}_k}\omega)<\frac{1}{k} \quad\Rightarrow\quad d(T^n\omega, T^{n+\hat{q}_k}\omega)=d(\omega, T^{\hat{q}_k}\omega)<\frac{1}{k}, \quad\forall n\in\{0,\ldots, k\hat{q}_k\}
	\end{equation*}
taking $q_k=\hat{q}_k$ we conclude that $\omega\in PRP(T)$, therefore $\{\Omega, T\}$ satisfies \textit{TRP}.\end{proof}


\begin{theorem}
\label{theorem 3}
(BD, \cite{damanikBase} theorem 3, p.651). \textbf{Spectral type of a generic quasi-periodic Schrödinger operator.} Consider the dynamical system $\{\mathbb{T}^d, T\}$, where $T:\mathbb{T}^d\rightarrow\mathbb{T}^d$ is the ergodic \textit{Shift} on $\mathbb{T}^d$: $T\omega=\omega+\alpha$, where $\alpha=(\alpha_1,\cdots, \alpha_d)$ and the set $\{1,\alpha_1,\dots,\alpha_d\}$ is independent over the rationals numbers. Then $\{\mathbb{T}^d, T\}$ satisfies \textit{TRP} and consequently, for a generic $\omega\in\mathbb{T}^d$ and $f\in C(\mathbb{T}^d)$, the operator:
\begin{equation}\label{Kap5Eq4}
		\begin{aligned}
			H_\omega: \ell^2(\mathbb{Z})&\rightarrow \ell^2(\mathbb{Z})\\
			\psi(n)&\mapsto \psi(n+1)+\psi(n-1)+f(T^n\omega)\psi(n)\\
			&=\psi(n+1)+\psi(n-1)+f(\omega+n\alpha)\psi(n)
		\end{aligned}
	\end{equation}
has purely continuous spectrum.
\end{theorem}

 \begin{proof}
Since the transformation $T:\mathbb{T}^d\rightarrow\mathbb{T}^d$ is ergodic, then $T$ is minimal on $\mathbb{T}^d$. Moreover, $T$ is an isometry, since for $\omega_1$ and $\omega_2$ in $\mathbb{T}^d$:
	\begin{equation*}
		\begin{aligned}
			d(T\omega_1,T\omega_2)&=d(\omega_1+\alpha, \omega_2+\alpha)=|\omega_1+\alpha-\omega_2-\alpha|=|\omega_1-\omega_2|=d(\omega_1,\omega_2)
		\end{aligned}
	\end{equation*}
As $T$ is minimal and an isometry on $\mathbb{T}^d$, in view of theorem \ref{thc} we conclude that the spectrum of the operator \eqref{Kap5Eq4} is purely continuous for a generic $\omega\in \mathbb{T}^d$. 
 \end{proof}

Now we study the system $\{\mathbb{T}^2, T\}$, where $T$ is the skew-shift: $T(\omega_1,\omega_2)=(\omega_1+2\alpha, \omega_1+\omega_2)$. For this purpose, we require the definition of the badly approximable subset of $\mathbb{T}$.

\begin{definition}
\label{mal aproximable}
	Let $\alpha\in\mathbb{T}=\mathbb{R}/\mathbb{Z}$. We define the constant $c(\alpha)$ as:
	\begin{equation*}
		c(\alpha)=\liminf_{q\rightarrow\infty}q\langle \alpha q\rangle\quad\text{where:}\quad \langle \alpha q\rangle=\text{dist$_\mathbb{T}(\alpha q, 0)$}=\min\{|\alpha q-p|: p\in\mathbb{Z}\}
	\end{equation*}

It is said that $\alpha\in\mathbb{T}$ is \textit{badly approximable} if $c(\alpha)>0$. Therefore, $\alpha\in\mathbb{T}$ is not badly approximable if there exists a subsequence $\{q_k\}_{k\in\mathbb{Z}_+}$ such that $\lim_{q_k\rightarrow\infty}q_k\langle \alpha q_k\rangle=0$.
\end{definition}

\begin{remark}
\label{medida mal aproximables}

The set of badly approximable numbers has Lebesgue measure equal to zero in $\mathbb{T}$ (Khinchin, \textit{\cite{khinchin} theorem 29, p.60}). Consequently, the set of numbers that are not a badly approximable is of full measure in $\mathbb{T}$.
\end{remark}




\begin{theorem}
\label{Teorema 4}
	\textit{(BD, \cite{damanikBase} theorem 4, p.651)}. 
 consider the dynamical system $\{\mathbb{T}^2, T\}$, where
$T:\mathbb{T}^2\rightarrow\mathbb{T}^2$ is the \textit{Skew Shift} operator: $T(\omega_1,\omega_2)=(\omega_1+2\alpha, \omega_1+\omega_2)$, with $\alpha\in\mathbb{R}\setminus\mathbb{Q}$. The following statements are equivalent.
 
	\begin{itemize}
		\item[(i)] $\alpha$ not badly approximable.
		\item[(ii)] $\{\mathbb{T}^2, T\}$ satisfies \textit{GRP}.
		\item[(iii)] $\{\mathbb{T}^2, T\}$ satisfies \textit{MRP}.
		\item[(iv)] $\{\mathbb{T}^2, T\}$ satisfies \textit{TRP}.
	\end{itemize}
\end{theorem}
\begin{proof}\textcolor{white}{Sea:}
	
	\underline{$(i)\Rightarrow (ii)$}:
Suppose that $\alpha$ is not badly approximable. Let us show that $\{\mathbb{T}^2, T\}$ satisfies \textit{GRP}, i.e., $\mathbb{T}^2=PRP(T)$. Since $\alpha$ is not badly approximable, then there exists a sequence $\{q_k\}_{k\in\mathbb{Z}_+}$, with $\lim_{k\rightarrow\infty}q_k=\infty$ such that.
	\begin{equation*}
		\lim_{k\rightarrow\infty}q_k\langle \alpha q_k\rangle=0
	\end{equation*}

Let $\omega=(\omega_1,\omega_2)\in\mathbb{T}^2$ and consider $k\in\mathbb{Z}_+$. To show that $\omega\in PRP(T)$, we shall construct a sequence $\hat{q}_k$ such that:
	\begin{equation}
 \label{Kap5Eq6}
		d(T^{n+\hat{q}_k}\omega,T^n\omega)<\frac{1}{k},\quad\text{for all $n$ such that:}\quad 0\leq n\leq\hat{q}_k
	\end{equation}
 According to the definition of $T$ in $\mathbb{T}^2$:
	
	\begin{equation}\label{Kap5Eq7}
		T^{n+\hat{q}_k}\omega-T^n\omega=(2\hat{q}_k\alpha, 2\hat{q}_k\omega_1+\hat{q}_k^2\alpha+2n\hat{q}_k\alpha-\hat{q}_k\alpha)
	\end{equation}
	We want to show that $\langle 2\hat{q}_k\alpha\rangle\rightarrow0$ and $\langle2\hat{q}_k\omega_1+\hat{q}_k^2\alpha+2n\hat{q}_k\alpha-\hat{q}_k\alpha\rangle\rightarrow0$. For each $k\in\mathbb{Z}_+$, let $\hat{q}_k=m_k q_k$, where $m_k\in\{1,\ldots,k+1\}$. The above implies that:
	\begin{equation*}
		T^{n+\hat{q}_k}\omega-T^n\omega=(2m_k q_k\alpha, 2m_k q_k\omega_1+(m_k q_k)^2\alpha+2nm_k q_k\alpha-m_k q_k\alpha)
	\end{equation*}

Since $\alpha$ is not badly approximable: $\langle 2m_k q_k\alpha\rangle\rightarrow0$. Likewise: $\langle (m_k q_k)^2\alpha \rangle\rightarrow0$, also $\langle m_k q_k\alpha\rangle\rightarrow0$ and $\langle 2nm_k q_k\alpha\rangle\rightarrow0$. It remains to show that $2m_k q_k\omega_1\rightarrow0$. For this purpose, we select $m_k$ from the set $\{1, \cdots, k+1\}$ such that $\langle m_k(2q_k\omega_1)\rangle<\varepsilon$. It follows that $T^{n+\hat{q}_k}\omega-T^n\omega\rightarrow0$ and consequently $d(T^{n+\hat{q}_k}\omega,T^n\omega)\rightarrow0$ for all $n$ such that that $0\leq n\leq \hat{q}_k$. It is concluded that $\omega\in PRP(T)$. Since the selection of $\omega$ was arbitrary, it is concluded that $\mathbb{T}^2 = PRP(T)$. It is concluded that $\mathbb{T}^2=PRP(T)$.	
 
	\quad
	
	\underline{$(ii)\Rightarrow (iii)$}: suppose that the dynamical system $\{\mathbb{T}^2,T\}$ satisfies \textit{GRP}, i.e. $PRP(T)=\mathbb{T}^2$. This implies that $\mu(PRP(T))=\mu(\mathbb{T}^2)=1>0$, then $\{\mathbb{T}^2,T\}$ satisfies \textit{MRP}.

	\quad
	
	\underline{$(iii)\Rightarrow (iv)$}: suppose $\{\mathbb{T}^2,T\}$ satisfies \textit{MRP}. Since $T$ is $\mu-$ergodic in $\mathbb{T}^2$ and the Lebesgue measure is strictly positive, then the system $\{\mathbb{T}^2,T\}$ satisfies \textit{TRP}.
	
	\quad
	
	\underline{$(iv)\Rightarrow (i)$}:
Suppose $\{\mathbb{T}^2,T\}$ satisfies \textit{TRP}. Then the set $PRP(T)$ is dense in $\mathbb{T}^2$. In particular, there exists $\omega\in\mathbb{T}^2$ such that $\{T^n\omega\}_{n\geq0}$ satisfies \textit{RP}. This means that for all $\varepsilon>0$
there exists a $q_k$ such that $d(T^{n+q_k}\omega, T^n\omega)<\varepsilon$. According to equation \eqref{Kap5Eq7}:

	\begin{equation}\label{Kap5Eq8}
		\langle 2q_k\omega_1+q_k^2\alpha+2nq_k\alpha-q_k\alpha\rangle<\varepsilon,\quad\text{for $n$ such that:}\quad 0\leq n\leq q_k
	\end{equation}
for $n = 0$ equation \eqref{Kap5Eq8} implies that $\langle 2q_k\omega_1+q_k^2\alpha-q_k\alpha\rangle<\varepsilon$. That is, $2q_k\omega_1+q_k^2\alpha-q_k\alpha\in\mathbb{T}$ is at a distance less than $\varepsilon$ from $0$. Then, for $\varepsilon>0$ arbitrarily small: $\langle 2nq_k\alpha\rangle=n\langle 2q_k\alpha\rangle$ for all $n$: $0\leq n\leq q_k$. Making $n=q_k$ we get $q_k\langle q_k\alpha\rangle<\varepsilon$. We conclude that $q_k\langle \alpha q_k\rangle\rightarrow0$, which by definition means that $\alpha$ is not badly approximable.
	\end{proof}
\begin{corollary}
\label{Teorema 4_1}
\textbf{Spectral type of Schrödinger operators with potential defined by skew-shift on $\mathbb{T}^2$}. Let $T:\mathbb{T}^2\rightarrow\mathbb{T}^2$ be the operator defined by $T(\omega_1,\omega_2)=(\omega_1+2\alpha, \omega_1+\omega_2)$, where $\alpha$ is not badly approximable. Then for a generic $\omega\in\mathbb{T}$ and $f:\mathbb{T}^2\rightarrow\mathbb{R}$, the operator:
	\begin{equation}\label{Kap5Eq20}
		\begin{aligned}
			H_\omega: \ell^2(\mathbb{Z})&\rightarrow \ell^2(\mathbb{Z})\\
			\psi(n)&\mapsto \psi(n+1)+\psi(n-1)+f(T^n\omega)\psi(n)\\
			&=\psi(n+1)+\psi(n-1)+f((\omega_1+2n\alpha, \omega_2+2n\omega_1+n(n-1)\alpha))\psi(n)
		\end{aligned}
	\end{equation}
	has purely continuous spectrum.
\end{corollary}
\begin{proof}
Suppose that $\alpha$ is not badly approximable. By theorem \ref{Teorema 4}, the system $\{\mathbb{T}^2, T\}$ satisfies \textit{TRP}. This in turn implies, as a consequence of \textbf{Theorem B}, that purely continuous spectrum is a generic property of $\{H_\omega\}_{\omega\in\mathbb{T}^2}$.
\end{proof}

In this paper we studied the conditions that allow us to conclude that purely continuous spectrum is a generic property of ergodic Schrödinger operators. The theorems studied, as pointed out by Boshernitzan and Damanik \cite{damanikBase}, support the notion that a repetition property is associated with the absence of point spectrum in Schrödinger operators. The topic covered in this paper can be expanded in multiple directions, for example:

\textit{Schrödinger operators in higher dimensions}. This domain has been covered by authors as Fan and Han \cite{fan}, who study the continuous spectrum of the Schrödinger operators in $\ell^2(\mathbb{Z}^d)$ for a measurable potential function $f: \mathbb{T}^d\rightarrow \mathbb{R}$.

\textit{Repetition properties and entropy}. A question of interest is to study what implications has the properties \textit{TRP} and \textit{MRP} on the entropy of a dynamical system. This issue was investigated by Huang et. al. \cite{huang}, who state that positive entropy of a dynamical system $\{\Omega, T\}$ implies the presence of point spectrum on a generic Schrödinger operator $\{H_\omega\}_{\omega\in\Omega}$. Finally, remains the question of neccesary conditions on the dynamical system $\{\Omega, T\}$ to guarantee the generic continuous spectrum of the family of Schrödinger operators $\{H_\omega\}_{\omega\in\Omega}$, topic that is framed in the theory of inverse problems in spectral theory.

\section*{Acknowledgements}
This paper is a product of the M.Sc. Thesis in Mathematics entitled \textit{Spectral analysis of ergodic Schrödinger operators}, made by the first author under the supervision of the second author, at \textit{Universidad Nacional de Colombia}. We would like to thank Professor Serafín Bautista for his guidance and recommendations, as well as Professor Leonardo Rendón for his suggestions to the work.

\printbibliography

@string{A = {Advances in Applied Mathematics}}

@string{C = {Communications in Mathematical Physics}}

@string{DI = {Journal of Dynamics and Differential Equations}}

@string{DU = {Duke Mathematical Journal}}

@string{S = {Journal of Spectral Theory}}

@ARTICLE{avilaDamanik,
	author={Artur Avila and David Damanik},
	title={Generic singular spectrum for ergodic schrödinger operators},
	journal= DU,
	volume={130},
	year={2005},
	pages={393-400}
}

@BOOK{axler2,
	author={Sheldon Axler},
	title={Linear Algebra Done Right},
	publisher={Springer-Verlag},
	address={New York},
	year={2015},
}

@ARTICLE{damanikBase,
	author={Michael Boshernitzan and David Damanik},
	title={Generic Continuous Spectrum for ergodic Schrödinger Operators},
	journal= C,
	volume={283},
	year={2008},
	pages={647-662}
}

@ARTICLE{fan,
	author={Yang Fan and Rui Han},
	title={Generic continuous spectrum for multi-dimensional quasiperiodic Schrödinger operators with rough potentials},
	journal= S,
	volume={8},
	year={2018},
	pages={1635-1645}
}

@ARTICLE{huang,
	author={Wen Huang and Leiye Xu and Yingfei Yi},
	title={Entropy of Dynamical Systems with Repetition Property},
	journal= DI,
	volume={23},
	year={2010},
	pages={683-693}
}

@BOOK{khinchin,
	author={Aleksandr Yakovlevich Khinchin},
	title={Continued Fractions},
	publisher={The University of Chicago Press},
	address={Chicago},
	year={1964},
}

@ARTICLE{simon,
	author={Barry Simon},
	title={Almost Periodic Schrödinger Operators: A Review},
	journal= A,
	volume={3},
	year={1982},
	pages={463-490}
}

\end{document}